\pdfoutput=1
\documentclass[sigconf]{acmart}

\copyrightyear{2022}
\acmYear{2022}
\setcopyright{acmcopyright}
\acmConference[ISSAC '22]{Proceedings of the 2022 International Symposium on Symbolic and Algebraic Computation}{July 4--7, 2022}{Villeneuve-d'Ascq, France}
\acmBooktitle{Proceedings of the 2022 International Symposium on Symbolic and Algebraic Computation (ISSAC '22), July 4--7, 2022, Villeneuve-d'Ascq, France}
\acmPrice{15.00}
\acmDOI{10.1145/3476446.3536170}
\acmISBN{978-1-4503-8688-3/22/07}
\usepackage{hyperref}
\usepackage{pgf,tikz}
\usepackage{graphicx}
\usepackage{amsmath}
\usepackage{upgreek}
\usepackage{xcolor}
\usepackage{booktabs}
\usepackage{todonotes}
\usepackage{multirow} 
\usepackage{tikz}
\usetikzlibrary{arrows,automata}
\usepackage{url}
\usepackage{algorithm}  
\usepackage{algorithmicx}
\usepackage[noend]{algpseudocode}
\usepackage{bm}
\usepackage{algorithm}
\usepackage{algorithmicx}

\algrenewcommand{\algorithmiccomment}[1]{\hfill\(\triangleright\)}
\usepackage{color}
\usepackage{lineno}
\usepackage{epstopdf}
\usepackage{subfigure}

\newtheorem{property}{Property}[section]
\newtheorem{remark}{Remark}[section]
\newcommand{\Res}{\mathrm{Res}}

 \renewcommand{\C}{\mathbb{C}}  %for arxiv
\newcommand{\Z}{\mathbb{Z}}

\newcommand{\Q}{\mathbb{Q}}
\newcommand{\R}{\mathbb{R}}
\newcommand{\A}{{\overline{\Q}}}
\newcommand{\e}{\mathrm{e}}
\newcommand{\h}{\mathrm{h}}
\renewcommand{\v}{\nu}

\renewcommand{\L}{\mathrm{L}}

\renewcommand{\log}{{\mathrm{log}_2}}
\newcommand{\M}{\mathrm{M}}

\renewcommand{\O}{\mathcal{O}}
\newcommand{\K}{\mathcal{K}}

\renewcommand{\imath}{{\rm i}}

\renewcommand{\H}{\mathrm{H}}
\renewcommand{\i}{{\imath}}
\newcommand{\disc}{\mathrm{disc}}
\renewcommand{\ln}{{\mathrm{log}}}
\renewcommand{\zeta}{{\mu}}

\newcommand{\tha}{{\vartheta}}

\AtBeginDocument{%
  }

\begin{document}
\fancyhead{}
%%
%% The "title" command has an optional parameter,
%% allowing the author to define a "short title" to be used in page headers.
\title{Explicit Bounds for Linear Forms in the Exponentials of Algebraic Numbers}

%%
%% The "author" command and its associated commands are used to define
%% the authors and their affiliations.
%% Of note is the shared affiliation of the first two authors, and the
%% "authornote" and "authornotemark" commands
%% used to denote shared contribution to the research.
\author{Cheng-Chao Huang}
\affiliation{
\institution{Institute of Software, Chinese Academy of Sciences}
% \city{Beijing}
\country{People's Republic of China}
}
\email{chengchao@nj.iscas.ac.cn}

%%
%% The "author" command and its associated commands are used to define
%% the authors and their affiliations.
%% Of note is the shared affiliation of the first two authors, and the
%% "authornote" and "authornotemark" commands
%% used to denote shared contribution to the research.

% \author{Ben Trovato}
% \authornote{Both authors contributed equally to this research.}
% \email{trovato@corporation.com}
% \orcid{1234-5678-9012}
% \author{G.K.M. Tobin}
% \authornotemark[1]
% \email{webmaster@marysville-ohio.com}
% \affiliation{%
%   \institution{Institute for Clarity in Documentation}
%   \streetaddress{P.O. Box 1212}
%   \city{Dublin}
%   \state{Ohio}
%   \country{USA}
%   \postcode{43017-6221}
% }

%%
%% By default, the full list of authors will be used in the page
%% headers. Often, this list is too long, and will overlap
%% other information printed in the page headers. This command allows
%% the author to define a more concise list
%% of authors' names for this purpose.
\renewcommand{\shortauthors}{Cheng-Chao Huang}

%%
%% The abstract is a short summary of the work to be presented in the
%% article.
\begin{abstract}
In this paper,
we study linear forms
\[
\lambda = \beta_1\e^{\alpha_1}+\cdots+\beta_m\e^{\alpha_m},
\]
where $\alpha_i$ and $\beta_i$ are algebraic numbers.
An explicit lower bound for the absolute value of $\lambda$ is proved, which is derived from
``th{\'e}or{\`e}me de Lindemann--Weierstrass effectif''
via constructive methods in algebraic computation. 
Besides, the existence of $\lambda$ with an explicit upper bound is established on the result of 
counting algebraic numbers. 
\end{abstract}

%%
%% The code below is generated by the tool at http://dl.acm.org/ccs.cfm.
%% Please copy and paste the code instead of the example below.
%%

%%
%% Keywords. The author(s) should pick words that accurately describe
%% the work being presented. Separate the keywords with commas.

%%
%% This command processes the author and affiliation and title
%% information and builds the first part of the formatted document.
\maketitle

\section{Introduction}

The study of transcendental number theory originated from the attention of Liouville~\cite{liouville1844classes} to a class of numbers, viz. transcendental numbers that satisfy no algebraic equation with integer coefficients.
In 1844, the existence of transcendental numbers is shown for the first time
by Liouville numbers constructed in the form of series,
which violate the lower bound for the approximation of algebraic numbers:
For an irrational algebraic number $x$ of degree $n$, there exists a constant $c(x) > 0$ such that $|x-\tfrac{p}{q}|<\tfrac{c(x)}{q^n}$ holds for all integers $p$ and $q>0$.
In 1874, Cantor~\cite{cantor1874ueber} 
also showed the existence of transcendental numbers by proving that the set of algebraic numbers is countable while the set of real numbers is uncountable.

In the same period, the exponential function became a topic of interest in transcendental number theory.
In 1873, Hermite~\cite{hermite1874fonction} proved the transcendence of Euler's number $\e$
by using auxiliary functions. 
Subsequently, Lindemann~\cite{lindemann1882ueber} proved that $\e^\alpha$ is transcendental for nonzero algebraic numbers $\alpha$.
In particular, the transcendence of $\pi$ is shown
since $\e^{\pi \imath}$ is algebraic,
which gave the negative answer to the ancient Greek question --- ``Squaring the Circle''.

More generally, transcendental number theory is also concerned with the algebraic independence of numbers. A set of numbers $\alpha_1,\ldots,\alpha_m$ is algebraically independent over a number field $\K$ if there is no nonzero polynomial $P$ in $n$ variables with coefficients in $\K$ such that $P(\alpha_1,\ldots,\alpha_m) = 0$.
Hereby, the transcendence of a number is a special case of $\Q$-algebraic independence with $m = 1$.
% In this aspect,
On this route,
in 1874, Hermite~\cite{hermite1874fonction} considered the independence 
of exponentials and
proved the $\Q$-linear independence of
$\e^{\alpha_1},\ldots,\e^{\alpha_m}$ for distinct rational numbers $\alpha_1,\ldots,\alpha_m$.
In 1882, a more general statement 
was sketched by Lindemann~\cite{lindemann1882ueber}
and was later further proved rigorously by Weierstrass~\cite{weierstrass1885lindemann},
which is known as
Lindemann–Weierstrass theorem.
% \begin{theorem}\label{thm:L1}
% If algebraic numbers $\alpha_1,\ldots,\alpha_m$ are 
% linearly independent over $\Q$,
% then $\e^{\alpha_1},\ldots,\e^{\alpha_m}$
% are algebraically independent over $\Q$.
% \end{theorem}
\begin{theorem}[Lindemann–Weierstrass]\label{thm:L2}
For any distinct algebraic numbers
$\alpha_1,\ldots,\alpha_m$ and
any nonzero algebraic numbers $\beta_1,\ldots,\beta_m$, we have
\[
\beta_1 \e^{\alpha_1}+\cdots+\beta_m \e^{\alpha_m} \neq 0.
\]
\end{theorem}

\paragraph{\bf Quantitative Aspects.}
Transcendental number theory also investigates transcendental numbers in a quantitative way.
Denote by $\lambda$ the linear forms of exponentials of algebraic numbers, i.e. $\lambda= \beta_1\e^{\alpha_1} + \cdots + \beta_m \e^{\alpha_m}$.
% \begin{equation}\label{eq:linearform}
%     \lambda = \beta_1\e^{\alpha_1} + \cdots + \beta_m \e^{\alpha_m}.
% \end{equation}
Notwithstanding $\lambda$ is nonzero, how far it is from zero is a challenging problem.
The study of the lower bound for such $\lambda$ has attracted the attention of mathematicians.

In 1929, Siegel~\cite{siegel2014einige} presented the lower bound for a special case of the $\Q$-linear forms $\lambda$ with integer exponents.
Specifically, let the linear forms $\lambda=a_0+a_1\e+\cdots +a_m\e^m$ with rational coefficients and $a=\max\{|a_0|,|a_1|,\ldots,|a_m|\}$.
Then there exists a positive number $c$ independent with $m$ and the coefficients, for any $a \ge a(m)$ we have
% \[
% |\lambda| \ge a^{-m-\tfrac{cm^2\ln(m+1)}{\ln\ln a}}.
% \]
\[
|\lambda| \ge a^{-m-{cm^2\ln(m+1)}/{\ln\ln a}}.
\]
In 1931, this result has been further improved  by Mahler~\cite{mahler1932approximation}
to
% \[
% |\lambda| \ge a^{-m-\tfrac{\mu+\varepsilon}{\ln\ln a}}
% \]
\[
|\lambda| \ge a^{-m-{(\mu+\varepsilon)}/{\ln\ln a}}
\]
for any $\varepsilon>0$ and $c>c(\varepsilon)$, where 
$\mu = m\ln(m! \e)-(m+1)\ln\frac{m+1}{\e}+1$.
% =m^2\ln m+\O(m^2).

% The research progress of the lower bound for liner forms lags behind that of polynomial forms.
% For the polynomial form,
% \begin{theorem}\label{thm:L1}
% If algebraic numbers $\alpha_1,\ldots,\alpha_m$ are 
% linearly independent over $\Q$,
% then $\e^{\alpha_1},\ldots,\e^{\alpha_m}$
% are algebraically independent over $\Q$.
% \end{theorem}

% \begin{theorem}[Lindemann–Weierstrass algebraic form]\label{thm:L1}
% For any algebraic numbers $\alpha_1,\ldots,\alpha_m$ linearly independent over $\Q$
% and any nonzero polynomial $P\in\A[x_1,\ldots,x_m]$,
% we have
% \[
% P(\e^{\alpha_1},\ldots,\e^{\alpha_m})\neq 0.
% \]
% \end{theorem}

Upon the other hand,
there is an equivalent formulation of Lindemann–Weierstrass theorem~\cite{baker1990transcendental},
which demonstrates  $\e^{\alpha_1},\ldots,\e^{\alpha_m}$
are algebraically independent over $\Q$
when
$\alpha_1,\ldots,\alpha_m$
are $\Q$-linear independent.
Namely, $P(\e^{\alpha_1},\ldots,\e^{\alpha_m})\neq 0$
for any nonzero polynomial $P\in\A[x_1,\ldots,x_m]$.
From this perspective,
the $\Q$-linear independence of the exponents
provides useful properties to establish the lower bound for $|P(\e^{\alpha_1},\ldots,\e^{\alpha_m})|$.

In 1932, Mahler~\cite{mahler1932approximation} gave a nontrivial lower bound
for such polynomials:
For a polynomial $P\in\Z[x_1,\ldots,x_m]$
of degree $\le d$ and height $\le H$,
there exists a constant
$H_0=H_0(d,\alpha_1,\ldots,\alpha_m)$ such that
\[
\ln |P(\e^{\alpha_1},\ldots,\e^{\alpha_m})|
\ge -cd^m\ln H
\]
for $H\ge H_0$ and a certain constant $c=c(\alpha_1,\ldots,\alpha_m)$.
In 1994, Ably~\cite{ably1994version} improved the result by using criteria for algebraic independence~\cite{philippon1986criteres}: 
For a polynomial $P$ with rational coefficient,
\[
\ln |P(\e^{\alpha_1},\ldots,\e^{\alpha_m})|
\ge -cd^m(\ln H+\exp(Cd^m\ln(d+1))),
\]
where 
$c=2^{4m^3+18m^2+25m+4}m^{m^2+m+2}(4D+m+1)^{m+2}D^{m^2+m+1}$,
$D=[\Q(\alpha_1,\ldots,\alpha_m):\Q]$, and $C=C(\alpha_1,\ldots,\alpha_m)$.

\paragraph{\bf Effective Aspects.}
The concept of ``effective results'' in number theory means their content can be effectively computed. 
The effective results are important and practical, in which the computable constants can actually be applied to solve other problems.
However, it is usually difficult, in general, to prove such effective results.
In indirect forms of proof, the implied constant often cannot be made computable, and the result based on Landau notation only assert the existence for such constant.

%So far, we noticed that the lower bounds aforementioned are all implicit.
Until 1998, Sert~\cite{sert1999version} further improved Ably's last result,
% For polynomials with algebraic coefficient,
and gave the first (to the best of our knowledge) effective version of the Lindemann--Weierstrass theorem (``th{\'e}or{\`e}me de Lindemann--Weierstrass effectif'' in French),
which is a breakthroughs in transcendental number theory.
The theorem provide
a lower bound for $|P(\e^{\alpha_1},\ldots,\e^{\alpha_m})|$,
%(see Theorem~\ref{thm:EffLW}),
which 
%is in an explicit form but 
can be computed in an effective manner. 
%complicated but computable.
% Here, a point to note is that it is not an explicit result,
% since the lower bound is related to some 
% parameters which are not represented in a explicit way via
% the basic constants, saying the number of terms or the size (degree or height) of algebraic numbers.

% In 1999, A. Sert provided proved the effective Lindemann--Weierstrass theorem
% (``th{\'e}or{\`e}me de Lindemann--Weierstrass effectif'' in French)~\cite{sert1999version}, 
% which established the first explicit lower bound for such $\Lambda = P(\e^{\alpha_1},\ldots,\e^{\alpha_m})$.
% Informally, it is triply exponentially small with respect to the sizes of the polynomial and these algebraic numbers.
% However, for linear forms, as far as we know, no explicit lower bound has been presented so far.
% Note that the equivalence transform
% from Theorem~\ref{thm:L2} to Theorem~\ref{thm:L1}
% is relatively clear, since a linear form can be 
% constructed naturally from a polynomial form.
% However, the reverse construction, i.e. from a linear form to a polynomial form, is not trivial, because of the strong restriction that $\alpha_1,\ldots,\ldots_m$ should be 
% linearly independent over $\Q$.
% Thus, even though we have the effective Lindemann--Weierstrass theorem, it is still difficult to derive a explicit lower bound for linear forms.

Back to the subject of this paper,
we focus on the linear forms
\[
\lambda = \beta_1\e^{\alpha_1}+\cdots+\beta_m\e^{\alpha_m},
\]
where $\alpha_i$ and $\beta_i$ are algebraic numbers of
bounded degrees and heights.
On the one hand, our purpose is to provide a nontrivial lower bound, by which
the absolute value of all such $\lambda$ are bounded from below.
One the other hand, we intend to give a nontrivial upper bound, 
there exists a nonzero $\lambda$ whose absolute value is bounded from above by it. Moreover, both the lower bound and the upper bound obtained are expressed in an explicit way.
For the former, we
construct an exponential polynomial $P$ equivalent to $\lambda$
such that the effective Lindemann--Weierstrass theorem can be applied.
To this end, we compute a linearly independent base of the field $\Q(\alpha_1,\ldots,\alpha_m)$, 
% with and bound the degree and the height of the basis elements.
%which is nontrivial.
and the main technique here is related to the computation over algebraic extension.
%By further deconstructing the functions in the effective Lindemann--Weierstrass theorem, saying the discriminant, we make the lower bound explicit.
For the latter,
we construct such $\lambda$ as the difference between 
two distinct linear forms.
By Dirichlet's pigeonhole principle,
we bound the difference
%the minimal distance between $\lambda_1$ and $\lambda_2$ 
from above, that is established
on the result of counting algebraic numbers. 
The main results are stated as follows:

\begin{theorem}[Main Result A]\label{thm:maina}
For any distinct algebraic numbers $\alpha_1,\ldots,\alpha_m$ and nonzero algebraic numbers $\beta_1,\ldots,\beta_m$ with the maximal degree $d$ and the maximal Weil height $h$,
we have
\begin{equation*} %\label{eq:eq_maina}
\begin{aligned}
\ln \,|\beta_1 \e ^{\alpha_1}+\ldots+\beta_m\e^{\alpha_m} | 
\ge 
 -  \e^{8\delta\zeta}
 - r m^{\delta}\e^{6\delta^2\zeta}
\big( 
m h + \tfrac{39}{164}\zeta  + \e^{R}\big),
\end{aligned}
\end{equation*}
in which 
% $\delta = d^{2m}$, $\zeta = md^{6m}(\tfrac{2h}{d}+3)$, and
\begin{itemize}
     \item[$\cdot$] $\delta = d^{2m}$,
     \item[$\cdot$] $\zeta = md^{6m}(\tfrac{2h}{d}+3)$,
     \item[$\cdot$] $R=r' m^{\delta}\e^{6\delta^2\zeta} + r'' m^{\delta}\e^{6\delta^2\zeta} (\ln m + 6\delta\zeta) + 72 \e^{2\delta\zeta}$,
    \item[$\cdot$] $r=82(\tfrac{9}{2})^\delta \delta^{3\delta+2}$,
    \item[$\cdot$] $r'= 12 (\tfrac{9}{2})^\delta \delta^{3\delta +2}
+ 16(1+6 \delta^2)(\tfrac{9}{2})^\delta \delta^{3\delta} \ln(9\delta^3) \\
{\qquad + \, 80(\tfrac{9}{2})^\delta \delta^{3\delta +4}
(1+6\delta) \zeta}$,
    \item[$\cdot$] $r''=16(1+6\delta^2)(\tfrac{9}{2})^\delta \delta^{3\delta}$.
\end{itemize}
\end{theorem}

% \begin{theorem}[Main Result B]\label{thm:mainb}
% For any distinct algebraic numbers $\alpha_1,\ldots,\alpha_m$ and nonzero algebraic numbers $\beta_1,\ldots,\beta_m$ with the maximal degree $d$ and the maximal Weil height $h$, 
% % Let $\lambda = \sum_i\beta_i \e ^{\alpha_m}$
% % where $\beta_1,\ldots,\beta_m$
% % and distinct $\alpha_1,\ldots,\alpha_m$
% % are algebraic numbers with the maximal degree $d$ and the maximal absolute multiplicative Weil height $H$.
% % For nonzero $\lambda$,
% \[
% \begin{aligned}
% & \min |\beta_1\e^{\alpha_1}+\cdots+\beta_m\e^{\alpha_m}| \\
% & \quad \le
% \e^{
% -\frac{1}{2}\ell\ln (N_1N_2)+\frac{1}{2}\ell\ln{\ell}
% +\ln{\ell} + 2+\ln 2\sqrt{2}},  
% %\quad \text{when}\;  N_1N_2 \ge 55 \ell,
% \end{aligned}
% \]
% in which $\ell=\lfloor \tfrac{m}{2} \rfloor$, and
% \begin{itemize}
% %    \item[$\cdot$] $\ell=\lfloor \tfrac{m}{2} \rfloor$, 
%     \item[$\cdot$] $N_1=\frac{1}{2} 6^{-d(d+1)}\e^{d(d+1)h}$, 
%     \item[$\cdot$] $N_2=\frac{1}{2}  72^{-\tfrac{1}{2}\sqrt{d}(\sqrt{d}+1)}   \e^{\tfrac{1}{2}\sqrt{d}(\sqrt{d}+1)h}$.
% \end{itemize}
% \end{theorem}

\begin{theorem}[Main Result B]\label{thm:mainb}
For any positive integers $m$ and $d$ and positive real number $h$, there exists a nonzero linear form $\lambda=\beta_1\e^{\alpha_1}+\cdots+\beta_m\e^{\alpha_m}$
where $\alpha_i$ and $\beta_i$ are algebraic numbers with the maximal degree $d$ and the maximal Weil height $h$, such that
% Let $\lambda = \sum_i\beta_i \e ^{\alpha_m}$
% where $\beta_1,\ldots,\beta_m$
% and distinct $\alpha_1,\ldots,\alpha_m$
% are algebraic numbers with the maximal degree $d$ and the maximal absolute multiplicative Weil height $H$.
% For nonzero $\lambda$,
\[
\ln \,|\lambda| 
\le
-\frac{1}{2}\ell \ln \frac{n_1n_2}{\ell}+\ln{m} + \ln {\sqrt{2}} + 2
\]
when $n_1\ge\ell$ and $n_2\ge 1$, in which
\begin{itemize}
    \item[$\cdot$] $\ell=\lfloor \tfrac{m}{2} \rfloor$,
    \item[$\cdot$] $n_1=\frac{1}{2} \Big( \tfrac{1}{6} \e^h \Big)^{d^2+d}$, and\; $n_2=\frac{1}{2} \Big(\tfrac{1}{6} \e^{\tfrac{1}{2}h-\tfrac{1}{2}\ln 2} \Big)^{d-\sqrt{d}}$.
\end{itemize}
\end{theorem}

This work is mainly to provide an explicit version of the
Lindemann\linebreak[0]--Weierstrass theorem, which is important to transcendental and computational number theory and closely related to theoretical computer science.
Note that the Lindemann--Weierstrass theorem plays an important role
in the study of decision problems for
the first-order logic or the continuous systems involving exponential functions.
It is utilised to prove the existence of the common root of two transcendental equations~\cite{anai2000deciding,bell2010continuous, mccallum2012deciding,chen2015continuous,chonev2016skolem,Xu2021QCTMC}
and the decidability of the sign of a transcendental numbers in such linear forms 
~\cite{AzizSSB96,feng2017precisely, xu2016multiphase, majumdar2021computability}.
It is also closely related to some fundamental problems in algebraic computation such as deciding the sign of polynomials at algebraic numbers. Even giving the lower bound for the absolute value of a polynomial at non-root rational numbers of bounded heights is challenging~\cite{garcia2018polynomial}.
Additionally, the Lindemann--Weierstrass theorem is generalised by Schanuel's conjecture~\cite{ax1971schanuel}, which is also widely applied in the field mentioned above~\cite{strzebonski2009real,huang2018positive,almagor2018minimal,majumdar2020decidability,chistikov2020big}.
Therefore, we believe the research and the results presented in this paper will potentially inspire or promote the study of these problems, especially from the perspective of computational complexity.

% \paragraph{\bf Organisation.}
% In Section 2, we provide some algebraic foundations.
% In Section 3, we give bounds for the primitive element of a algebraic extension.
% In Section 4, we give bounds for expressing polynomials of the algebraic numbers in the primitive element.
% In Section 5, we give a lower bound for $|\lambda|$.
% In Section 6, we give a upper bound for $\min|\lambda|$. 
% %In Section 7, we draw a conclusion.

\section{Preliminary}
In this section, we briefly review some algebraic notions, definitions, and properties,
especially for some measurements of the size of polynomials and algebraic numbers.
Meanwhile, we also deduce some basic results,
which will be used in the analysis afterwards.

% \subsection{Algebraic Number}
% % Denote by 
% % $\C$ the field of complex numbers,
% % $\Q$ the field of rational numbers,
% % and $\A$ the field of algebraic numbers.
% % an algebraic closure of $\Q$, viz.
% Let $f\in\C[x]$ be a polynomial of degree $d$:
% \begin{equation}\label{eq:polyf}
% f(x)=a_d x^d + a_{d-1}x^{d-1}+\cdots+a_1x+a_0.
% \end{equation}
% Define the height of $f$ by
% \[
% \H(f) = \max\{|a_0|,|a_1|,\ldots,|a_d|\},
% \]
% and the $2$-norm of $f$ by 
% \[
% |f| = \sqrt{a_0^2+a_1^2+\cdots+a_d^2}.
% \]
% For a multivariate polynomial $f\in\C[x_1,\ldots,x_m]$,
% define the length $\L(f)$ of $f$ as
% \[ \L(f) = |a_0|+|a_1|+\cdots+|a_d|. \]
% %the sum of the modulus of its coefficients.
% \begin{property}\label{prop:len}
% For polynomials $f_1$ and $f_2$ over $\C$,
% \[
% \L(f_1+f_2) \le \L(f_1)+\L(f_2).
% \]
% \end{property}
% A complex number $\alpha$ is algebraic if it is a root of a polynomial in $\Z[x]$.
% Let $\A$ denote the field of algebraic numbers.
% % Algebraic numbers form a field $\A$ with the operations. 
% A polynomial $f$ is the defining polynomial of $\alpha$, if it has least degree and co-prime integer coefficients, and vanishes at $\alpha$. Here $\alpha$ is said to be of the same degree and height of $f$. 

% %\textcolor{blue}{brief intro of algebraic extension (degree...).}
% % Note that the degree of $p_\alpha$ is upper bounded by $||p||$, and the length is upper bounded by $2^{||p||}$.
% % For algebraic numbers $\alpha$ and $\beta$,
% % we have $\v(p_{\alpha\beta})\le\v(p_\alpha p_\beta)$~\cite{cohen2013course}.

\subsection{Polynomials and Algebraic Numbers}
Let $f\in\C[x]$ be a polynomial
\begin{equation}\label{eq:polyf}
f(x)=a_d x^d + a_{d-1}x^{d-1}+\cdots+a_1x+a_0.
\end{equation}
where $a_d\neq 0$ is called the leading coefficient of $f$ and $d=\deg(f)>0$ is the degree of $f$.
The size of $f$ can be measured
by the $p$-norm of the coefficient vector
$(a_d,\ldots,a_0)$.
% \[
% ||f||_p = \left(\sum_{i=0}^d |a_i|^p\right)^{1/p}.
% \]
Specifically, 
the $\infty$-norm is called
the \emph{height} of the polynomial:
\[
\H(f)=\max\{|a_0|,\linebreak[0]\ldots,|a_d|\},
\]
the $1$-norm is called the \emph{length} of the polynomial:
\[L(f)=|a_0|+\cdots+|a_d|.
\]
The $2$-norm denoted by $\|f\|_2$ is the ordinary Euclidean length of the coefficient vector.
For a multivariate polynomial $f\in\C[x_1,\ldots,x_m]$,
the $p$-norm can be also defined in a similar way
according to the coefficient vector identified with $f$. 

% \begin{property}\label{prop:len}
% Let $f$ and $g$ be two polynomials over $\C$. Then
% \[
% \L(f+g) \le \L(f)+\L(g).
% \]
% \end{property}

% For a commutative ring $\mathcal{R}$, a semi-norm
% is any function $\v:\mathcal{R}\to \R_{\ge 0}$ 
% satisfying the following conditions \cite{collins1974minimum} for all $a,b\in\mathcal{R}$:
% \[
% \begin{aligned}
% & \v(a)=0 \iff a=0,\\
% & \v(a-b) \le \v(a)+\v(b),\\
% & \v(ab) \le \v(a)\v(b).
% \end{aligned}
% \]

For a commutative ring $\mathcal{R}$, the semi-norm~\cite{collins1974minimum} 
can be defined by a function $\v:\mathcal{R}\to \R_{\ge 0}$.
By further defining $\v(\sum_{i=0}^d a_i x^i)=\sum_{i=0}^d \v(a_i)$,
it can be extended to 
the semi-norm for the polynomial ring $\mathcal{R}[x]$
and generalised to the multivariate polynomials in
$\mathcal{R}[x_1,\ldots,x_m]$ via the induction on $m$. 
Then, the semi-norm of the resultant of
two polynomials can be bounded.
\begin{theorem}[Theorem 2 of \cite{collins1974minimum}]
Let $f$ and $g$ be polynomials over a commutative ring $\mathcal{R}$ with semi-norm $\v$. Then
\[
\v(\Res(f,g))  \le \v(f)^{\deg (g)} \v(g)^{\deg(f)}.
\]
\end{theorem}
Note that the length $L$ is a semi-norm function.
Thus, by considering $f,g\in\Z[x,y]$ as univariate polynomials in $y$ over the commutative ring $\Z[x]$, we have the following corollary.
\begin{corollary}\label{coro:semi}
Let $f$ and $g$ be polynomials in $\Z[x,y]$. Then
\[
\L(\Res_y(f,g)) \le \L(f)^{\deg_y(g)} \L(g)^{\deg_y(f)}.
\]
\end{corollary}

Let $\A$ denote the field of algebraic numbers.
A complex number $\alpha$ is algebraic if it is a root of a polynomial in $\Z[x]$.
% Algebraic numbers form a field $\A$ with the operations. 
The defining polynomial of $\alpha$ is the unique polynomial $f$ of least degree, which vanishes at $\alpha$ and has co-prime integer coefficients. 
Here, $\alpha$ is said to be of the same degree and height of $f$. 
If $f$ is monic, viz. the leading coefficient is $1$, then $\alpha$ is called an algebraic integer.

%\textcolor{blue}{brief intro of algebraic extension (degree...).}
% Note that the degree of $p_\alpha$ is upper bounded by $||p||$, and the length is upper bounded by $2^{||p||}$.
% For algebraic numbers $\alpha$ and $\beta$,
% we have $\v(p_{\alpha\beta})\le\v(p_\alpha p_\beta)$~\cite{cohen2013course}.

\subsection{Absolute Logarithmic Height}
Let $\K$ be a number field. 
% For $v\in M_k$, the completion of $\K$ at $v$ is denoted by $\K_v$.
For $\alpha\in\K$, define
$
H_{\K}(\alpha)= \prod_{v\in M_{\K}}  \max \linebreak[0] \{ 1,|\alpha|_v \}^{d_v}$,
where $M_\K$ is the set of normalised absolute value of $\K$ 
and $d_v$ is the degree of the completion of $\K$ at $v$ over $\R$~\cite{waldschmidt2013diophantine}. When $\alpha$ is an algebraic number, the \emph{absolute logarithmic height} (or \emph{Weil height}) $\h:\A \to [0,\infty)$ is defined by
\begin{equation}\label{eq:weilh}
\h(\alpha)=\frac{1}{[\K:\Q]}\ln H_\K(\alpha),
\end{equation}
and $\e^{\h(\alpha)}$ is called the absolute multiplicative height.
Note that the Weil height does not depend on the choice of the number field $\K$ containing $\alpha$.
Notably, for a nonzero rational number $\alpha=a/b$ in lowest terms, $\h(\alpha)=\ln \max\{|a|,|b|\}$.
\begin{property}[Property 3.3 of \cite{waldschmidt2013diophantine}]\label{prop:h}
Let $\alpha_1$ and $\alpha_2$ be two algebraic numbers, then
\[
\begin{aligned}
&\h(\alpha_1\alpha_2)\le \h(\alpha_1) + \h(\alpha_2),\\
&\h(\alpha_1+\alpha_2)\le \ln 2 + \h(\alpha_1) + \h(\alpha_2).
\end{aligned}
\]
For any algebraic number $\alpha\neq 0$ and any integer $n$,
$
\h(\alpha^n) = |n|\h(\alpha).
$
\end{property}

\begin{lemma}[Lemma 3.11 of \cite{waldschmidt2013diophantine}]\label{lem:hH}
For an algebraic number $\alpha$ of degree $d$, we have
\[
\frac{1}{d}\ln \H(\alpha)-\ln 2
\le \h(\alpha)
\le \frac{1}{d}\ln \H(\alpha) + \frac{1}{2d}\ln (d+1).
\]
\end{lemma}

For any vector $\alpha = (\alpha_1,\ldots,\alpha_m)\in\K^{n}$, define
$
H_\K(\alpha) = \prod_{v\in M_\K} \linebreak[0]\max \{1, \max_{i=1}^m|\alpha_i|_v\}^{d_v}$,
and its Weil height can be further defined as Equation~\eqref{eq:weilh} 
when $\alpha_1,\ldots,\alpha_m$ are algebraic.
% When the elements in the vector $\alpha$ are algebraic numbers,  Weil height we furthers define
% \[
% \h(\alpha) = \frac{1}{[\K : \Q]} \ln H_\K (\alpha).
% \]
By this definition,
% By this definition of the $H_\K$, we have
\[
H_\K(\alpha) \le \prod_{i=1}^m \prod_{v\in M_\K} \max\{1, |\alpha_i|_v\}^{d_v} = \prod_{i=1}^m H_\K (\alpha_i).
\]
Then, we have the following lemma:
\begin{lemma} \label{lem:hw_vec}
Let $\alpha =(\alpha_1,\ldots,\alpha_m)\in\A^{m}$. Then
\[
\h(\alpha)\le \sum_{i=1}^m \h (\alpha_i).
\]
\end{lemma}

Let $f\in\C[x]$ be a nonzero polynomial of
degree $d$ in the form
$f(x)=a_d\prod_{i=1}^d(x-\alpha_i)$.
The Mahler's measure of $f$ is defined by
\[
\M(f) = |a_d| \prod_{i=1}^d \max \{1,|\alpha_i|\}.
\]
For each algebraic number $\alpha$ with defining polynomial $f\in\Z[x]$, its Mahler's measure is define
by $\M(\alpha)=\M(f)$.
\begin{lemma}[Lemma 3.10 of \cite{waldschmidt2013diophantine}]\label{lem:hmah}
Let $\alpha$ be an algebraic number of degree $d$. Then
$\h(\alpha) = \frac{1}{d}\ln \M(\alpha)$.
\end{lemma}

% \begin{proof}
% It follows from that
% \[
% \L(f) \le 2^d \M(f) \le 2^d \L(f).
% \]
% \end{proof}

% \begin{corollary}\label{coro:broot}
% Let $f\in\Z[x]$ be a polynomial of degree $d$ and Weil height $h$.

% \textcolor{\black}{need be refine}
% Then each root $\alpha$ of $f$ satisfies 
% \[
% |\alpha| \le \e^{d(h+\ln{2})} .
% \] 
% \end{corollary}

\begin{corollary}\label{coro:broot}
Let $f\in\Z[x]$ be a polynomial of degree $d$ and Weil height $h$.
Then each root $\alpha$ of $f$ satisfies 
$
|\alpha| \le \e^{dh}$.
\end{corollary}
The corollary follows from the definition of Mahler's measure, since $h$ can be expressed as
\[h = \frac{1}{d}\ln |a_d|
+\frac{1}{d}\sum_{i=1}^d \max\{0,\ln |\alpha_i|\}.
\]
%Then the corollary follows from the fact that $|a_d|\ge 1$.

\begin{lemma}\label{lem:hl}
For an algebraic number $\alpha$ of degree $d$, we have
\[
\frac{1}{d}\ln \L(\alpha)-\ln 2
\le \h(\alpha)
\le \frac{1}{d}\ln \L(\alpha).
\]
\end{lemma}
The lemma follows from
the property that
$2^{-d}\L(f) \le \M(f) \le  \L(f)$~\cite{mahler1960application}.
Note that Property~\ref{prop:h},
Lemma~\ref{lem:hH}, and Lemma\ref{lem:hl} will be applied frequently in the proofs afterwards. 
For simplicity, sometimes we may omit the references to them.

\subsection{\bf Discriminant of Algebraic Extension}
Let $f\in\K[x]$ be a polynomial of the form~\eqref{eq:polyf}.
The discriminant of $f$ is defined by
\[
\mathrm{disc}(f)= \tfrac{(-1)^{d(d-1)/2}}{a_d}\Res(f,f').
\]
% By applying Hadamard's inequality on
% the resultant, we have the following lemma:

\begin{lemma}[Theorem 1 of \cite{mahler1964inequality}]
Let $f\in\C[x]$ be a polynomial of degree $d$. Then   
% \[
% |\mathrm{disc}(f)| \le d^{d-1} \v(f)^{2n-1}\].
\[
|\mathrm{disc}(f)|\le d^d\M(f)^{2d-2}.
\]
\end{lemma}
By Lemma~\ref{lem:hmah}, we immediately have the following corollary.
\begin{corollary}\label{coro:discrF}
Let $f\in\Z[x]$ be a polynomial of degree $d$ and Weil height $h$. Then   
% \[
% |\mathrm{disc}(f)| \le d^{d-1} \v(f)^{2n-1}\].
\[
|\mathrm{disc}(f)| \le d^d \e^{2d(d-1)h}.
\]
\end{corollary}

% \begin{lemma}\label{lem:discr}
% Note that $\K=\Q(\gamma)$, 
% and let $\gamma_1,\ldots,\gamma_d$
% be the conjugates of $\gamma$,
% i.e. the distinct roots 
% of the defining polynomial of $\gamma$.
% According to Corollary~\ref{coro:Lagrange2}, we have
% \[
% |\gamma_i| \le \v(\gamma).
% \]
% Then,
% \[
% \begin{aligned}
% \Delta_P=& \prod_{i<j}(\gamma_i-\gamma_j)^2\\
% \le & \prod_{i<j}(|\gamma_i|+|\gamma_j|)^2\\
% \le & (2v)^{2d(d-1)}\\
% = & 2^{\O(d^2\log v)}.
% \end{aligned}
% \]
% \end{lemma}

%%%%%%%%%%%%%%%%%%%%%%%%%%%%%%%%%%
% about extension

% A field extension $\mathcal{L}/\K$ is called a simple extension if there exists an primitive element $\theta$ in $\mathcal{L}$ with $L=K(\theta)$. 
% The finite algebraic extension $\Q(\alpha_1,\ldots,\alpha_m)$ is a simple extension.
% The integral basis of $\Q(\theta)$ is a subset such that every element in it can be uniquely expressed as a linear combination of basis elements with integer coefficients.
% The degree of $\Q(\theta)$, denoted by $[\Q(\theta)/\Q]$, is the number of the elements in its integral basis. 

% For a simple extension $\Q(\theta)$, 
% let $\theta_1,\ldots,\theta_m$
% be the conjugates of the primitive element $\theta$,
% i.e. the distinct roots 
% of the defining polynomial of $\theta$.
% By Corollary~\ref{Lagrange2}, the discriminant
% \[
% \begin{aligned}
% \Delta_{\Q(\theta)}=& a_n^{2n-2}\prod_{i<j}(\theta_i-\theta_j)^2\\
% \le & \v(p_\theta)^{2n-2}\prod_{i<j}(|\theta_i|+|\theta_j|)^2\\
% \le & \v(p_\theta)^{2n-2} (2v(p_\theta))^{\deg(\theta)(\deg(\theta)-1)}.\\
% \end{aligned}
% \]

\section{Primitive Element of Algebraic Extension}\label{sec:prim}
% Let $\alpha_1,\ldots,\alpha_m\in\A$ be $m$ algebraic numbers, and
% $\K=\Q(\alpha_1,\ldots,\alpha_m)$ be the algebraic extension generated by them.
% By the primitive element theorem, there exists a primitive element $\theta\in\mathcal{A}$
% with degree $n=[\K / \Q]$ such that $\Q(\theta)=\Q(\alpha_1,\ldots,\alpha_m)$. 
% Namely, every element $x$ of the algebraic extension $\K$
% can be written uniquely in the $\Q$-linear form
% $x=c_{n-1}\theta^{n-1}+\cdots+c_1\theta+c_0$
% of powers of $\theta$.

Recall the primitive element theorem, which is a fundamental result about field extension:
\begin{theorem}[Theorem 4.1.8 of \cite{cohen1993course}]\label{thm:simple}
Let $\K$ be a finite extension of $\Q$ with $[\K:\Q]=n$.
Then there exists an element $\theta\in\K$ such that 
\[
\K = \Q(\theta).
\]
The $\theta$ is called a primitive element,
whose defining polynomial is an irreducible polynomial of degree $n$.
\end{theorem}

Let $\alpha_1,\ldots,\alpha_m$
be $m$ distinct algebraic numbers.
% of degree $\ge 2$.
According to Theorem~\ref{thm:simple},
the algebraic extension $\Q(\alpha_1,\ldots,\alpha_m)$
generated by them is a simple extension
with a certain primitive element $\theta$.
%We denote by $\theta$ a certain primitive element of the extension field.
% Then its powers $1,\theta,\theta^2,\ldots,\theta^{d_\theta-1}$ form
% a $\Q$-linearly independent base, where $d_\theta$ is the degree of $\theta$.
We intend to give the upper bounds for the degree and the Weil height of such $\theta$, respectively.
To this end, the classic algebraic algorithm SIMPLE is invoked to construct the primitive element.
%$\theta$ from the generators $\alpha_1,\ldots,\alpha_m$.
% We first consider the situation that the algebraic extension is generated by two algebraic numbers,
% and then generalise the result by constructing the primitive element recursively from multiple generators.

% The procedure to construct the basis is as below:
% \begin{itemize}
%     \item using the algorithm SIMPLE to construct a primitive element $\theta$ of the 
% the algebraic extension.
% \item based on the factorisation over $\Q(\theta)$,
% we factor the defining polynomial of $\alpha_i$ to obtain the 
% coefficients of the linear expression.
% \end{itemize}

\subsection{For Two Generators}
In 1982, Loos~\cite{loos1982computing} presented the algorithm SIMPLE for computing a simple extension over $\Q$. 
%to construct
%the primitive element of an algebraic extension (See Algorithm~\ref{alg:simple}).
Let $\alpha$ and $\beta$ be two algebraic numbers represented by their defining polynomials and isolating intervals, respectively.
The primitive element $\theta$ of the algebraic extension $\Q(\alpha,\beta)$ can be constructed.

Denote by $A$ and $B$ the defining polynomials of $\alpha$ and $\beta$.
In the algorithm SIMPLE,
the resultant $ r(x,t) = \Res_y(A(x-ty),B(y))$ is constructed such that it has the root $\theta=\alpha+t\beta$, 
the constant $t$ can be further chosen 
%as $t^*$ 
to ensure that
$r(x,t)$ has no multiple roots,
and $r(x,t)$ is exactly the defining polynomial of $\theta$.
%is obtained in Step~3 by substituting $t=t^*$ into $r(x,t)$. 
The upper bound for such constant $t$ is given by the lemma below.
\begin{lemma}[Theorem~4.10 of~\cite{yokoyama1989computing}]\label{lem:int_t}
There is at least an integer $t$
among $\deg(A)\deg(B)(\deg(B)-1)$
distinct integers such that $r(x,t)$ is square-free.
\end{lemma}

Following the constructive procedure in the algorithm SIMPLE,
we can give the bounds for the degree and the Weil height of such $\theta$.
\begin{lemma}\label{lem:simple2}
Let $\alpha$ and $\beta$ be two algebraic numbers represented by their defining polynomials $A$ and $B$ with the maximal degree $d$ and the maximal Weil height $h$.
Then, there exists a primitive element $\theta$ of the algebraic extension $\Q(\alpha,\beta)$ 
%with its defining polynomial $C$ 
satisfying
\[
\begin{aligned}
    % & \deg(\theta) \le d^2, \\
    % & \h(\theta) \le   2dh + 3d \ln d +2d\ln 2.\\
 \deg(\theta) \le d^2 
 \quad\text{and}\quad
 \h(\theta) \le   2dh + 3d \ln d +2d\ln 2.
%  &\log \v(C) \le n_1l_2+n_2l_1 + n_1n_2\log(n_1^2n_2) + n_2.
\end{aligned}
\]
\end{lemma}
\begin{proof}
%For a certain $t\in \Z$, $A(x-ty)$ is a bivariate polynomial in $x$ and $y$.
We consider the positive integer $t$ as a fixed constant.
Then, the bound for the degree of $\theta$ is obvious since
$\deg(r(x,t))\le \deg_x(A(x-ty))\deg_y(B(y))$.
Without loss of generality,
we consider the polynomial $A$ of the form~\eqref{eq:polyf}.
For the $k$th term of $A(x-ty)$, we have
$\L(a_k(x-ty)^k) \le  |a_k|\sum_{i=0}^k \tbinom{k}{i}  t^i 
= |a_k|(t+1)^k$.
Let $L$ be the maximal length of $A$ and $B$.
Then, by the subadditivity of the length function, we have
\[
\begin{aligned}
    \L(A(x-ty)) & \le  \sum_{k=0}^{d} |a_k|(t+1)^k 
\le  \sqrt{\sum_{k=0}^{d} |a_k|^2 \cdot \sum_{k=0}^{d} (t+1)^{2k}} \\
& \le  L \cdot \sqrt{ \frac{(t+1)^{2({d}+1)}-1}{(t+1)^2-1} }
\le  L (t+1)^d(1+\tfrac{1}{t})^{\tfrac{1}{2}}.
\end{aligned}
\]
From Lemma~\ref{lem:int_t}, we have $1\le t\le d^2(d-1)$.
By Corollary~\ref{coro:semi}, 
% \begin{equation}\label{eq:ineq_ex_2}
% \ln \L(C) \le 2d\ln L + d^2\ln (t+1) + \frac{1}{2}d\ln(1+\tfrac{1}{t}).
% \end{equation}
\[
\begin{aligned}
\ln \L(r) & \le  2d\ln L + d^2\ln (t+1) + \frac{1}{2}d\ln(1+\tfrac{1}{t})\\
& \le 2d\ln L + d^2\ln(d^3-d^2+1) + \tfrac{1}{2}d\ln 2\\
& \le 2d \ln L + 3d^2 \ln d. \qquad\text{($d\ge 2$)}
\end{aligned}
\]
With $\deg(\theta)\ge d$, the proof is completed by Lemma~\ref{lem:hl}.
\end{proof}

\begin{remark}
In Section~5 of \cite{loos1982computing}, 
Loos also gave the bound for the length
of $\theta$ without the proof. 
That is $\L(\theta)=\O(L^{2d})$, in which
$L$ denotes the maximal length of $\alpha$
and $\beta$.
By Lemma~\ref{lem:hl}, it is not hard
to check this bound is consistent with 
the result we obtained. 
However, it cannot be further used to infer explicit results since it is established on the Landau notation.
\end{remark}

\subsection{For Multiple Generators}
Towards the situation of
the algebraic extension generated 
by multiple algebraic numbers, 
we construct the primitive element  
by recursively applying the algorithm SIMPLE on two primitive elements of algebraic extensions generated
by fewer algebraic numbers.
Intuitively, we demonstrate the procedure by a tree structure in Figure~\ref{fig:tree}.
By regarding the $m$ generators $\alpha_1,\ldots,\alpha_m$ as the leaf nodes,
a full binary tree with the minimal depth can be constructed
from these elements to the root $\theta$,
where each non-leaf node is a primitive element computed
from its two children.
By applying Lemma~\ref{lem:simple2}
along this procedure,
we can further bound
the degree and the Weil height
of the primitive element $\theta$ constructed.
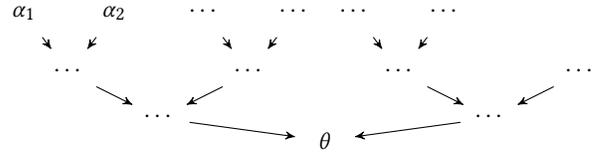
\begin{figure}[ht]
\centering
\begin{tikzpicture}[->,>=stealth',auto,node distance=1cm,
thin,inner sep=1pt]
\node[state,draw=white](s11){$\theta$};
\node[state,draw=white](s12)[above of=s11,xshift=-2.2cm,yshift=-0.7cm]{$\cdots$};
\node[state,draw=white](s21)[above of=s11,xshift=2.2cm,yshift=-0.7cm]{$\cdots$};
\node[state,draw=white](s13)[above of=s12,xshift=-1.2cm,yshift=-0.4cm]{$\cdots$};
\node[state,draw=white](s23)[above of=s12,xshift=1.2cm,yshift=-0.4cm]{$\cdots$};
\node[state,draw=white](s32)[above of=s21,xshift=-1.2cm,yshift=-0.4cm]{$\cdots$};
\node[state,draw=white](s31)[above of=s21,xshift=1.2cm,yshift=-0.4cm]{$\cdots$};
\node[state,draw=white](s14)[above of=s13,xshift=-.6cm,yshift=-0.2cm]{$\alpha_1$};
\node[state,draw=white](s25)[above of=s13,xshift=.6cm,yshift=-0.2cm]{$\alpha_2$};
\node[state,draw=white](s35)[above of=s23,xshift=-.6cm,yshift=-0.2cm]{$\cdots$};
\node[state,draw=white](s34)[above of=s23,xshift=.6cm,yshift=-0.2cm]{$\cdots$};
\node[state,draw=white](s43)[above of=s32,xshift=-.6cm,yshift=-0.2cm]{$\cdots$};
\node[state,draw=white](s53)[above of=s32,xshift=.6cm,yshift=-0.2cm]{$\cdots$};
% \node[state,draw=white](s52)[above of=s31,xshift=-.5cm,yshift=-0.1cm]{$\cdots$};
% \node[state,draw=white](s41)[above of=s31,xshift=.5cm,yshift=-0.1cm]{$\cdots$};

\draw[<-](s11)edge[]node{}(s12);
\draw[<-](s11)edge[]node{}(s21);
\draw[<-](s12)edge[]node{}(s13);
\draw[<-](s12)edge[]node{}(s23);
\draw[<-](s21)edge[]node{}(s32);
\draw[<-](s21)edge[]node{}(s31);
\draw[<-](s13)edge[]node{}(s14);
\draw[<-](s13)edge[]node{}(s25);
\draw[<-](s23)edge[]node{}(s35);
\draw[<-](s23)edge[]node{}(s34);
\draw[<-](s32)edge[]node{}(s53);
\draw[<-](s32)edge[]node{}(s43);
% \draw[<-](s31)edge[]node{}(s52);
% \draw[<-](s31)edge[]node{}(s41);
\end{tikzpicture}
\caption{The primitive element $\theta$ of an algebraic extension can be constructed recursively. The full binary tree is not unique since the order of the generators as the leaf nodes is arbitrary.}\label{fig:tree}
\end{figure}

\begin{lemma}\label{lem:multi_simple}
    Let $\alpha_1,\ldots,\alpha_m$ be $m$ algebraic numbers
    with the maximal degree $n$ and the maximal Weil height $h$.
    Then there exists a primitive element $\theta$ of the algebraic extension $\Q(\alpha_1,\ldots,\alpha_m)$
%which is defined by the polynomial $C$ with
satisfying
    \[
    \begin{aligned}
        % &\deg(\theta)\le d^{2m},\\
        % &\h(\theta) \le md^{2m}(\tfrac{2h}{d}+3).
\deg(\theta)\le d^{2m}
\quad\text{and}\quad
\h(\theta) \le md^{2m}(\tfrac{2h}{d}+3).
%        \le 2md^{2m-1}h+3md^{2m}.
    \end{aligned}
    \]
\end{lemma}

\begin{proof}
We indicate the layer index of the nodes by $0,1,\ldots,{\ell}$
from the layer of leaves to the root,
and denote by $d_i$ and $h_i$ the maximal degree and the maximal Weil height of the elements at the layer of index $i$. 
By Lemma~\ref{lem:simple2}, we have $d_i \le d_{i-1}^2 \le d^{2^i}$ and
\[h_i \le 2d_{i-1}h_{i-1} + D_{i-1}, 
\quad\text{where}\quad D_i =3 d_i \ln d_i + 2 d_i \ln 2.\]
So, the upper bound for the Weil height of $\theta$ can be inferred as
$
h_{\ell} \le  
2d_{{\ell}-1}(\ldots (2d_1(2d_0h_0+D_0)+D_1)\cdots)+D_{{\ell}-1}$,
which can be expanded as the sum of $h'$ and $h''$ where
\[
 \begin{aligned}
h' & = 2^{\ell} h_0 \prod_{i=0}^{{\ell}-1} d_{i} 
 \le  2^{\ell} h  d^{\sum_{i=0}^{{\ell}-1} 2^i} 
=  2^{\ell} h d^{2^\ell -1},\\
h'' & = \sum_{i=0}^{{\ell}-1}
\Big (D_i 2^{{\ell}-i-1} \linebreak \prod_{k=i+1}^{{\ell}-1}d_k \Big) \\
& \le 
d^{2^\ell} 2^{\ell-1}
\sum_{i=0}^{\ell-1}
\Big(
\tfrac{3\ln d}{d^{2^i}} +\tfrac{2\ln 2}{d^{2^i} 2^i}
\Big) \\
% = 
& \le 
3d^{2^{\ell}}2^{{\ell}-1}. \qquad\text{($d\ge 2$)}
 \end{aligned}
\]
Then, the proof is completed by
${\ell}=\lceil \log m \rceil \le \log m +1$.
\end{proof}

\begin{remark}
It should be noted that the rational univariate reduction can also be regarded as an effective version of the primitive element theorem which has been around since the work of Kronecker. The set of the defining polynomial of the generators can be considered as a system, and some similar bounds for the primitive element can be derived from the $u$-resultant (see \cite{saugata2003geometry} and Gap Theorem in Section 3.3 of \cite{canny1988complexity}).
\end{remark}

\begin{remark}
Note that in the proofs
of Lemma~\ref{lem:simple2} and \ref{lem:multi_simple},
$d\ge 2$ is applied 
to scale the inequalities.
% For $d=1$, the generators of the algebraic extension are rational.
% Then, $\theta=1$ can be considered as 
% the primitive element, for which
However, we omit this condition to make the results more generalised
because it is easy to check that the bounds obtained are still valid when $d=1$
by considering $\theta=1$ as 
the primitive element. 
\end{remark}

\section{factorisation over Algebraic Extension}\label{sect:fact}
% To convert $\lambda$ into the form of exponential polynomial
% with $\Q$-linearly independent exponents,
% we intend to represent the $\alpha_1,\ldots,\alpha_m$
% by $\Z$-linear combinations of a base,whose coefficients can be further bounded.
%For the primitive element $\theta$ of algebraic extension $\Q(\alpha_1,\ldots,\alpha_m)$,
% For the primitive element $\theta$ of degree
% $d_\theta$,
% its powers $1,\theta,\theta^2,\ldots,\theta^{d_\theta-1}$
% must be $\Q$-linearly independent.
% Otherwise, $\theta$ can be defined by a polynomial
% of degree less than $d_\theta$.
To convert the linear forms $\lambda$ into the form of exponential polynomial
with $\Q$-linearly independent exponents,
we intend to represent the exponents $\alpha_1,\ldots,\alpha_m$
by linear combinations of a basis
consisting of the powers of a primitive element
of $\Q(\alpha_1,\ldots,\alpha_m)$.
To this aim, the technique utilised here 
is to factorise 
the defining polynomial $f_i(x)$
of each exponent $\alpha_i$
into monic irreducible factors
over the algebraic extension.

In 1983, Lenstral~\cite{lenstra1983factoring} presented
an algorithm for factorising polynomial over algebraic extension,
which is a generalisation of the factorisation over $\Q$~\cite{lenstra1982factoring}. 
Specifically, for an algebraic integer $\theta$,
any polynomial $f$ over $\Q(\theta)$ can be factorised
into monic irreducible factors in $\tfrac{1}{T}\Z(\theta)[x]$,
where $T$ is a positive integer.
% The polynomials in $\frac{1}{T}\Z(\theta)[x]$ of degree $\le D$
% form a lattice in $(\tfrac{\Z}{T})^{(D+1) d_\theta}$,
Note that a monic polynomial $g$ in $\frac{1}{T}\Z(\theta)[x]$ can be represented by 
$a_{d,0}x^d+\sum_{i=0}^{d-1}\sum_{j=0}^{d_\theta-1} a_{i,j} \theta^j x^i$, which is identified by the $(d_{\theta}d+1)$-dimensional coefficient vector.
Then, the upper bound for the $2$-norm of the monic factors
has been provided in the proof of Proposition~3.11 in \cite{lenstra1983factoring}, which can be
stated as below.
\begin{theorem}\label{thm:fact}
Let $\theta$ be an algebraic integer with the defining polynomial $p_\theta$ of degree $d_\theta$.
Let $f\in \tfrac{1}{t}\Z(\theta)[x]$ with positive integer $t$ be a monic polynomial of degree $d$.
Then, $f$ can be factorised into the monic irreducible factors in $\tfrac{1}{T}\Z[\theta][x]$, where $T = t|\mathrm{disc}(p_\theta)|$
and the $2$-norm of each factor of degree $\le D$ is at most
\[
\H(f)\big( 2(d+1) d_\theta^3 (d_\theta - 1)^{d_\theta -1} \tbinom{2D}{D} \big)^{\tfrac{1}{2}}
\|p_\theta\|_2^{2(d_\theta -1)}|\mathrm{disc}(p_\theta)|^{-\tfrac{1}{2}}.
% \O(n^3 d_\theta^3 + n^2 d_\theta^3\log( d_\theta |p_\theta|)
%  + n^2 d_\theta^2 \log(t H(f)),
\]
% in which $\mathrm{disc}(F)$ is the discriminant of $F$.
\end{theorem}

Consider an algebraic integer $\tha$ which is
a primitive element of $\Q(\alpha_1,\ldots,\alpha_m)$.
The defining polynomial of each $\alpha_i$
can be factorised over $\Q(\tha)$.
For each $\alpha_i$,
the monic irreducible factor in the form
$x-\tfrac{1}{T}p_i(\tha)$ with $p_i(x)\in \Z[x]$ can be obtained. 
Then, we can further provide the upper bounds for $T$ and the length of each $p_i$.

\begin{lemma}\label{lem:multi_simple_2}
    Let $\alpha_1,\ldots,\alpha_m$ be $m$ algebraic numbers
    with the maximal degree $d$ and the maximal Weil height $h$.
    Then there exists a primitive element $\tha$ of
    $\Q(\alpha_1,\ldots,\alpha_m)$
    such that $\alpha_i=\tfrac{1}{T}p_i(\tha)$ for $i=1,\ldots,m$, where $T$ is a positive integer and each $p_i(x)\in\Z[x]$ is a polynomial
    satisfying
    \[
     T \le \e^{4md^{8m}(\tfrac{2h}{d}+3)}
     \quad\text{and}\quad
     \L(p_i) \le \e^{6md^{8m}(\tfrac{2h}{d}+3)}.
    \]
\end{lemma}

\begin{proof}
Let $\theta$ be the primitive element 
of $\Q(\alpha_1,\ldots,\alpha_m)$ constructed
in Section~\ref{sec:prim}.
Note that the defining polynomial $p_\theta$ of $\theta$ 
may not be monic. 
So, we further choose the algebraic integer $\tha=l_\theta \theta$
as a new primitive element,
where $l_\theta$ is the leading coefficient
of $p_\theta$.
Denote by $p_\tha$ its defining polynomial of degree $d_\tha$ and Weil height $h_\tha$.
Then, it is obvious that
\begin{equation}
d_\tha  = d_\theta \le d^{2m},
\end{equation}
By Property~\ref{prop:h}, we can further infer that
\begin{equation}\label{ineq:h_tha}
\begin{aligned}
h_\tha & = \h(l_\theta \theta)\\
& \le 
\ln l_\theta + \h(\theta) \le \ln \H(\theta) + \h(\theta)\\
& \le
d^{2m}(md^{2m}(\tfrac{2h}{d}+3)+\ln2) + md^{2m}(\tfrac{2h}{d}+3)\\
& \le 2md^{4m}(\tfrac{2h}{d}+3).
\qquad\text{($d\ge 2$)}
\end{aligned}
\end{equation}
Note that the upper bound for $\h(\tha)$
is also valid when $d=1$.

Let $f_i$ be the defining polynomial of each $\alpha_i$ with the leading coefficient $l_i$.
We choose $t=\prod_{i=1}^m l_i$ such that
every $f_i / l_i$ is a monic polynomial in $\tfrac{1}{t}\Z(\tha)[x]$.
% The first part of the conclusion follows from Corollary~\ref{coro:discrF} and Lemma~\ref{lem:multi_simple}.
By Theorem~\ref{thm:fact}, 
all the monic irreducible factors each $f_i / l_i$ are in $\frac{1}{T}\Z(\tha)[x]$ with
the common $T = t|\disc(p_\tha)|$.
Then, by Corollary~\ref{coro:discrF}, we have 
\[
\begin{aligned}
T & \le \prod_{i=1}^m\H(f_i) \cdot |\disc(p_\tha)|\\
& \le
\e^{md(h+\ln 2)} d_\tha^{d_\tha}\e^{2d_\tha(d_\tha-1)h_\tha}
      \\
& \le \e^
{4md^{8m}(\tfrac{2h}{d}+3)
- (4md^{6m}(\tfrac{2h}{d}+3) - 2md^{2m}\ln{d} - md(h+\ln 2) )}.
\end{aligned}
\]

Note that $f_i(\alpha_i)/l_i=0$, which implies $f_i/l_i$ has the monic irreducible factor $x-\frac{1}{T}p_i(\tha)$ of degree as most $d_\tha$.
It is also obvious that $p_i$ contains at most $d_\tha$ terms. Thus,
\[
\begin{aligned}
\L(p_i) & \le {d_\tha}^{\tfrac{1}{2}}  \|p_i\|_2 
 \le {d_\tha}^{\tfrac{1}{2}}   \|x-\frac{1}{T}p_i(\tha)\|_2 T,
\end{aligned}
\]
and by Theorem~\ref{thm:fact}, we have
$\ln \L(p_i)\le \ln L_1 + \ln L_2$ where
\[
\begin{aligned}
\ln L_1 
& = \ln (\|p_\tha\|_2^{2(d_\tha -1)}  |\disc(p_\tha)|^{-\tfrac{1}{2}} T ) \\
& = 
2(d_\tha -1) \ln \|p_\tha\|_2
+ \ln(t|\disc(p_\tha)|^{\tfrac{1}{2}} )\\
& \le 
 6md^{8m}(\tfrac{2h}{d}+3) 
  -(4md^{6m}(\tfrac{2h}{d}+3) - 2(d^{2m}-1)d^{2m}\ln 2),\\
\ln L_2
& = \ln \H(f_i) +
\tfrac{1}{2}\ln (2(d+1)d_\tha^4(d_\tha-1)^{d_\tha-1})+ \tfrac{1}{2}\ln\tbinom{2d_\tha}{d_\tha}\\
& \le \ln \H(f_i) 
+ 2 \ln{d_\tha}
+ \frac{1}{2} \ln (2(d+1)(d_\tha-1)^{d_\tha-1})
+d_{\tha} \ln 2
\\
& \le  d(h+\ln 2) + 4m \ln d 
+ \frac{1}{2}\ln(2d+2) \\
& \qquad + \frac{1}{2}(d^{2m}-1)\ln(d^{2m}-1)
+ d^{2m} \ln 2.
\end{aligned}
\]
Then, the result follows by combining $\ln L_1$ and $\ln L_2$.
\end{proof}

% \begin{remark}\label{remk:base}
% Note that by the algorithm SIMPLE we can also obtain two polynomials $C_1(\theta)$
% and $C_2(\theta)$ representing
% $\alpha$ and $\beta$.
% However, for multiple generators, 
% it is hard to bound the coefficients (especially for their common denominator) in the representing polynomial $P_i(x)$ for each $\alpha_i$ 
% after iteratively composing.
% \end{remark}

%%%%%%%%%%%%%%%%%%%%%%%%%%%%%%%%%%%%%%%%%%%%%%%%%%%%%%%%%%%%%%%%%
\section{Proof of Main Result A}
For the linear forms $\lambda$,
we intend to rewrite it into the form as $P(\e^{\alpha_1'},\ldots,\e^{\alpha_{M}'})$,
where $P[x_1,\ldots,x_{M}]$ is a multivariate polynomial over
$\Q(\beta_1,\ldots,\beta_m)$,
and the new exponents ${\alpha_1'},\ldots,{\alpha_{M}'}$ are
$\Q$-linear independent.
Then, the effective Lindemann--Weierstrass theorem can be  applied to analyse the lower bound for $|\lambda|$.

\subsection{Constructing the Exponential Polynomial}
%In this section, we construct the multivariate polynomial $P$
%according to the primitive element $\theta$, then establish the lower bound for $|\lambda|$.
Following Section~\ref{sect:fact},
the algebraic integer $\tha$ is a primitive element 
of $\Q(\alpha_1,\ldots,\alpha_m)$,
and the defining polynomials of
$\alpha_1,\ldots,\alpha_m$ can be factorised into monic irreducible factors in $\frac{1}{T}\Z(\tha)[x]$.
As mentioned previously, for each exponent $\alpha_i$, there exists a polynomial $p_i\in\Z[x]$ 
such that $\alpha_i=\tfrac{1}{T}p_i(\tha)$.
Namely, $\frac{1}{T},\frac{1}{T}\tha,\ldots,\frac{1}{T}\tha^{d_\tha-1}$ form a basis, by which every $\alpha_i$ can be $\Z$-linearly expressed.

Without loss of generality, for $i=1,\ldots,m$, we write 
\[
p_i(x)=c_{i,0}+c_{i,1}x+\cdots\linebreak[0]+c_{i,d_\tha-1}x^{d_\tha-1},\]
by which each $\e^{\alpha_i}=\e^{\tfrac{1}{T}p_i(\tha)}$ can be expanded as
\[
\e^{\alpha_i}=\e^{\tfrac{1}{T} c_{i,0}} \e^{\tfrac{1}{T} c_{i,1} \tha} \cdots
\e^{\tfrac{1}{T} c_{i,d_\tha-1} \tha^{d_\tha-1}}.
\]
Correspondingly, we further define the multivariate
polynomial
\[
P_i(x_0,x_1,\ldots,x_{d_\tha-1})=x_0^{c_{i,0}}x_1^{c_{i,1}}\cdots x_{d_\tha-1}^{c_{i,d_\tha-1}},
\]
such that
$\e^{\alpha_i}=P_i(\e^{\tfrac{1}{T}},\e^{\tfrac{1}{T}\tha},\ldots,\e^{\tfrac{1}{T}\tha^{d_{\tha}-1}})$.
Then, the linear forms can be expressed as
% \[\lambda=\beta_1 P_1(\e^{\tfrac{1}{T}},\linebreak[0]\e^{\tfrac{1}{T}\tha},\ldots,\e^{\tfrac{1}{T}\tha^{d_{\tha}-1}})+\cdots+\beta_mP_m(\e^{\tfrac{1}{T}},\linebreak[0]\e^{\tfrac{1}{T}\tha},\ldots,\e^{\tfrac{1}{T}\tha^{d_{\tha}-1}}).
% \]
\[\lambda=\sum_{i=1}^m\beta_i P_i\Big(\e^{\tfrac{1}{T}},\linebreak[0]\e^{\tfrac{1}{T}\tha},\ldots,\e^{\tfrac{1}{T}\tha^{d_{\tha}-1}}\Big).
\]
% and $\lambda=P(\e^{\tfrac{1}{T}},\linebreak[0]\e^{\tfrac{1}{T}\tha},\ldots,\e^{\tfrac{1}{T}\tha^{d_{\tha}-1}})$ where $P=\beta_1 P_1+\cdots+\beta_2P_2$.
% Note that each $P_i$ is essentially Laurent polynomials, since there may exist
% negative coefficients in $p_i$.
Note that each $P_i$ here is still a Laurent polynomial, which may contain negative exponents caused by the negative coefficients in $p_i$.
% A technique to handle it is
% difference substitution~\cite{yang2005solving},
% which is used in \cite{huang2018positive}
% to transform the base such that given algebraic numbers
% can be $\Z^+$-linearly expressed by the new base.
% However, it is hard to provide a bound for these coefficients produced.
To make these exponents positive, we product each $P_i$ by a common monomial with a sufficiently large degree.
Specifically, for $j=0,\ldots,d_\tha-1$,
denote by $\hat c_j$ the maximal of the absolute values of
negative coefficients of $x^j$ in $p_1,\ldots,p_m$, viz. $\hat c_j = \max_{i=1}^{m} |c_{i,j}| [c_{i,j}<0]$ where the brackets denote the indicator function. 
Then, we write
% $P'_i= C P_i$ with 
\begin{equation}\label{eq:polyP}
C(x_0,x_1,\ldots,x_{d_\tha-1})=\prod_{j=0}^{d_\tha-1} x_j^{ \hat c_j} \quad\text{and}\quad
P=C\sum_{i=1}^m \beta_i P_i.
\end{equation}
Now, $P$ is a multivariate polynomial in $\Z[x_0,x_1\ldots,x_{d_\tha-1}]$,
%the linear combination of 
%$CP_1,\ldota $ $ P=\beta_1 C P_1+\cdots,\beta_m C P_m$.
% Then, we have
% \[
% \begin{aligned}
% \lambda & = C^{-1}(\e^{\tfrac{1}{T}},\e^{\tfrac{1}{T}\tha},\ldots,\e^{\tfrac{1}{T}\tha^{d_{\tha}-1}}) P(\e^{\tfrac{1}{T}},\e^{\tfrac{1}{T}\tha},\ldots,\e^{\tfrac{1}{T}\tha^{d_{\tha}-1}}).
% % \sum_{i=1}^m \beta_i  P_i(\e^{\tfrac{1}{T}},\e^{\tfrac{1}{T}\tha},\ldots,\e^{\tfrac{1}{T}\tha^{d_{\tha}-1}})\\
% % & =  C^{-1}(\e^{\tfrac{1}{T}},\e^{\tfrac{1}{T}\tha},\ldots,\e^{\tfrac{1}{T}\tha^{d_{\tha}-1}}) \sum_{i=1}^m \beta_i   P'_i(\e^{\tfrac{1}{T}},\e^{\tfrac{1}{T}\tha},\ldots,\e^{\tfrac{1}{T}\tha^{d_{\tha}-1}})\\
% % & =  C^{-1}(\e^{\tfrac{1}{T}},\e^{\tfrac{1}{T}\tha},\ldots,\e^{\tfrac{1}{T}\tha^{d_{\tha}-1}}) P(\e^{\tfrac{1}{T}},\e^{\tfrac{1}{T}\tha},\ldots,\e^{\tfrac{1}{T}\tha^{d_{\tha}-1}}).
% \end{aligned}
% \]
and $\lambda$ can be represented by the production
of $\chi'$ and $\chi''$ where
%, i.e., $\lambda = \chi' \chi''$ where
\[
\begin{aligned}
\chi' & = C^{-1}\Big(\e^{\tfrac{1}{T}},\e^{\tfrac{1}{T}\tha},\linebreak[0]\ldots,\e^{\tfrac{1}{T}\tha^{d_{\tha}-1}}\Big),\\
\chi'' & = P \Big(\e^{\tfrac{1}{T}},\e^{\tfrac{1}{T}\tha},\ldots,\e^{\tfrac{1}{T}\tha^{d_{\tha}-1}}\Big).
\end{aligned}
\]
% $\chi'=C^{-1}(\e^{\tfrac{1}{T}\tha},\linebreak[0]\ldots,\e^{\tfrac{1}{T}\tha^{d_{\tha}-1}})$ and
% $\chi''=P(\e^{\tfrac{1}{T}\tha},\ldots,\e^{\tfrac{1}{T}\tha^{d_{\tha}-1}})$. 
In what follows, we will give the lower bounds for $\chi'$ and $\chi''$ and combine them together to deduce the result.

\subsection{Analysing the Bounds}
% \noindent{\bf Lower bound for $\mathbf{\chi'}$:}\\
{\bf Lower Bound for $\chi'$:} 
Noting that $C(x_0,\ldots,x_{d_\tha-1})$ is a monomial,
we can analyse the lower bound for $\chi'$ in a straightforward way.
%We first give the upper bound for the degree of $C$. 
Noting the degree of $C$ in each $x_j$ comes from the absolute value of a certain coefficient in $p_i$,
we have
\[
\deg_{x_j}(C) = \hat c_j  \le \max_i \L(p_i).
\]
% By Theorem~\ref{thm:sep}, we have
% \[
% \begin{aligned}
% T=|\mathrm{disc}(f_i)| &\ge \prod_{i<j}(\tha_i-\tha_j)^2 \\
% &\ge \mathrm{sep}(f_i)^{n(n-1)} \\
% &\ge \left( \sqrt{3}n^{-\tfrac{n}{2}-1} \v(f_i)^{-n+1} \right)^{n(n-1)}.
% \end{aligned}
% \]
Moreover, by Corollary~\ref{coro:broot}, we bound the absolute value of $\tha$ as
\[
|\tha| \le \e^{\deg(\tha)\h(\tha)}
\le\e^{ 2md^{6m}(\tfrac{2h}{d}+3)}.
\]
Then, with the fact that $T \ge 1$, we have
\[
\begin{aligned}
  \ln \, |\chi'|  & =  
  \ln \, \Big| C^{-1}\Big(\e^{\tfrac{1}{T}},\e^{\tfrac{1}{T}\tha},\ldots,\e^{\tfrac{1}{T}\tha^{d_{\tha}-1}}\Big) \Big| \\
%\ln |\lambda '| 
  & \ge  -\ln \prod_{j=0}^{d_\tha-1} \e^{ \tfrac{\hat c_j}{T}|\tha|^j} 
  =   - \sum_{j=0}^{d_\tha-1}  \tfrac{\hat c_j}{T}|\tha|^j \\
  & \ge -  \max_{i} \L(p_i)
  \cdot \sum_{j=0}^{d_\tha-1} |\tha|^j. 
\end{aligned}
\]
For the case that $|\tha| \ge 2$, we have
%$ |\tha|\le \e^{md^4m}$ and $d_{\tha} \le d^{2m}$, we have
\[
\sum_{j=0}^{d_\tha-1} |\tha|^j \le |\tha|^{d_\tha} \le \e^{2md^{8m}(\tfrac{2h}{d}+3)}.
\]
For the case that $|\tha|< 2$, it is obviously that $\sum_{j=0}^{d_\tha-1} |\tha|^j < 2^{d_\tha}$, that means the upper bound above is still valid. 
Then, with the bound for $\L(p_i)$ from Lemma~\ref{lem:multi_simple_2}, we immediately have
\begin{equation}\label{eq:ld1}
\ln |\chi'| \ge - \e^{8md^{8m}(\tfrac{2h}{d}+3)}.
\end{equation}

% by which an effective lower bound for $|P(\e^{\alpha_1},\ldots,\e^{\alpha_m})|$
% is given under the independence condition
% viz. $\alpha_1,\ldots,\alpha_m$ are $\Q$-linearly independent.

{\bf Lower Bound for $\chi''$:} 
Recall that $\chi''$ is in the form of exponential polynomial with $\Q$-linearly independent exponents. So, the effective Lindemann--Weierstrass theorem can be applied
to analyse its lower bound.
To be self-contained, we present this theorem as below (with slight adjustment on the notations).
\begin{theorem}[Sert, Theorem~3 of \cite{sert1999version}]\label{thm:EffLW}
Let $\alpha_1,\ldots,\alpha_M \in \K$ be linear independent over $\Q$, and  $\alpha=(\alpha_1,\ldots,\alpha_M)$.
Let $P\in\K[X_1,\ldots,X_M]$ be a nonzero polynomial of degree $\le d_P$ ($d_P\ge 1$) with the coefficients $\beta=(\beta_1,\ldots,\beta_{m})$.
    Then
    \[
    \begin{aligned}
    & \ln  \left| P(\e^{\alpha_1},\ldots,\e^{\alpha_M}) \right| \\
    & \ge -r{d_P}^M\big( \h(\beta) + \tfrac{39}{328D}\ln|\Delta_\beta|+\e^{r'{d_P}^M
      + r''{d_P}^M\ln {d_P} + 72 \hat\alpha} \big) ,
    \end{aligned}
    \]
    in which $D$ is the degree of $\K$ over $\Q$,
    $\Delta_\beta$ is the discriminant of $\Q(\beta_1,\ldots,\beta_M)$, and
    \begin{itemize}
    % \item[$\cdot$] $D$ is the degree of $\K$ over $\Q$,
    % \item[$\cdot$] $\Delta_\beta$ is the discriminant of $\Q(\beta_1,\ldots,\beta_m)$,
    \item[$\cdot$] $\hat \alpha = \max\{1, \max_i |\alpha_i|\}$,
    \item[$\cdot$] $r = 41  \cdot 2^{-M+1}  3^{2M} M^M D^{M+1}$,
    \item[$\cdot$] 
    $r' =   2^{-M+2}  3^{2M+1} M^M D^{M+1}\\
     {\qquad +\, (1+6D)2^{-M+4}3^{2M} M^M D^M \ln (9MD)}\\
     {\qquad +\, 2^{-M+4}3^{2M} M^M D^{M+1} (1+6M) \h(\alpha)}$,
    \item[$\cdot$] $r''=(1+6D) 2^{-M+4} 3^{2M} M^M D^{M}$.
    \end{itemize}
\end{theorem}

Then, according the bounds in Lemma~\ref{lem:multi_simple} and~\ref{thm:fact}, we give the explicit upper bounds
related to the following objects, that appear in 
effective Lindemann--Weierstrass theorem as parameters.

\noindent $\bullet\,$ For the algebraic extensions:\\
Let $\K$ be the algebraic extension of $\Q$
generated by the exponents
$\frac{1}{T},\frac{1}{T}\tha,\frac{1}{T}\tha^2,\ldots,\frac{1}{T}\tha^{d_\tha-1}$
and all the coefficients in $P$,
which implies $\K \subseteq \Q(\alpha_1,\ldots,\alpha_m,\beta_1,\ldots,\beta_m)$. 
By Lemma~\ref{lem:multi_simple}, we have
\[
D = [\K:\Q]\le d^{4m}.
\]
Furthermore, let $\vartheta'$ (of degree $d_\vartheta'$ and Weil height $h_\vartheta'$) be the primitive element of $\Q(\beta_1,\ldots,\beta_m)$.
% Then, we have
% $\deg(\vartheta')\le (n)^{2m}$ and
% $\log\v(p_{\vartheta'})\le 6mn^{2m}(l+\log m\log n)$.
Then, the discriminant of $\Q(\vartheta')$ is equal to the discriminant of the defining polynomial of $\vartheta'$. 
By Corollary~\ref{coro:discrF} and Lemma~\ref{lem:multi_simple}, we have
\[
\begin{aligned}
\ln |\Delta_\beta|& = \ln |\disc(\Q(\vartheta'))| \\
%& \ln |\mathrm{disc}(\Q(\beta_1,\ldots,\beta_m))| =\ln |\disc(\Q(\vartheta'))| \\
& \le  d_{\vartheta'} \ln d_{\vartheta'} + 2d_{\vartheta'}(d_{\vartheta'}-1)h_{\vartheta'}\\
& \le 2md^{2m}\ln d + 2d^{2m}(d^{2m}-1)md^{2m}(\tfrac{2h}{d}+3)\\
% & \le 2md^{2m}\ln d + 2md^{6m}(\tfrac{2h}{d}+3)- 2md^{4m}(\tfrac{2h}{d}+3)\\
& \le 2md^{6m}(\tfrac{2h}{d}+3).
\end{aligned}
\]

\noindent $\bullet\,$  For the polynomial:\\
Let P be the polynomial we constructed in Equation~\eqref{eq:polyP}.
Denote by $M$ the number of variables of $P$, and we have
\[
M\le d_\tha \le d^{2m}.
\]
Recall the form of $P$, in which each $c_{i,j}$ is the coefficient of $x^j$ in $p_i(x)$.
Then, we can define 
$
i^*=\arg\max_i \sum_j c_{i,j}
$,
which implies $P_{i^*}$ has the maximal degree
among $P_1,\ldots,P_m$.
Thus, with the fact that $C$ and each $P_i$ are monomials, we have
\[
\begin{aligned}
d_P 
& \le \deg(C) + \deg(P_{i^*}) \\
& = \sum_{j=0}^{d_\tha-1} \hat c_j + \sum_{j=0}^{d_\tha-1} c_{i^*,j}\\
& = \sum_{j=0}^{d_\tha-1} ( \hat c_j + c_{i^*,j} ) [c_{i^*,j}<0] + \sum_{j=0}^{d_\tha-1} c_{i^*,j}[c_{i^*,j}>0]\\
& \le m \max_i\L(p_{i}). \\
\end{aligned}
\]
The last inequality above holds
because every $\hat c_j$ from $p_{i^*}$ will be eliminated by the corresponding $c_{i^*,j}$, that implies
\[
\sum_{j=0}^{d_\tha-1} ( \hat c_j + c_{i^*,j} ) [c_{i^*,j}<0] \le \sum_{i\neq i^*} \L(p_{i}).
\]
Thus, by Lemma~\ref{lem:multi_simple_2}, we have
\[
d_P \le m \e^{6md^{8m}(\tfrac{2h}{d}+3)}.
\]
%}
Note that the coefficients in $P$
are exactly the coefficients $\beta_1,\ldots,\beta_m$
appearing in the linear forms $\lambda$.
So, by Lemma~\ref{lem:hw_vec}, we have
%\textcolor{\black}{
\[
\h(\beta) \le \sum_{i=1}^m \h(\beta_i) 
\le m h.
\]
%}

\noindent $\bullet\,$ For the exponents:\\
Let $\alpha = (\tfrac{1}{T}, \tfrac{1}{T}\tha,\ldots,\tfrac{1}{T}\tha^{d_\tha-1})$.
For the case $|\tha| \ge 1$, we have
% \[
%  \max_i |\alpha_i| \le |\tha|^{d_\tha-1}
% \le \e^{\deg(\tha^{d_\tha-1})\h(\tha^{d_\tha-1})},
% \]
% for which we have $\deg(\tha^{d_\tha-1})\le d_\tha$ since $\tha^{d_\tha-1}\in\Q(\tha)$,
% and $\h(\tha^{d_\tha-1})\le (d_\tha-1)\h(\tha)$ by Property~\ref{prop:h}.
% Then, with Lemma~\ref{lem:multi_simple}, we immediately have
\[
\max_i |\alpha_i|  \le |\tha|^{d_\tha}
\le \e^{2md^{8m}(\tfrac{2h}{d}+3)}.
\]
Obviously, the bound above also holds when $|\tha|<1$.
Moreover, according to Property~\ref{prop:h}, we have
\[
\h(\tfrac{1}{T}\tha^i)\le \h(T) + i\h(\tha) = \ln T + ih_\tha.
\]
Then, by combining this with Lemma~\ref{lem:hw_vec}, we have
\[
 \begin{aligned}
\h(\alpha) & \le \sum_{i=0}^{d_\tha-1} \h(\tfrac{1}{T}\tha^i)
%\le d_\tha \h(T) + \h(\tha) \sum_{i=0}^{d_\tha-1} i\\
\le 
d_\tha \ln T + \frac{d_\tha(d_\tha-1)}{2} h_\tha \\
& \le
4md^{10m}(\tfrac{2h}{d}+3)
+ d^{2m}(d^{2m}-1) md^{4m}(\tfrac{2h}{d}+3)\\
& \le 5 md^{10m}(\tfrac{2h}{d}+3).
\end{aligned}
\]
%%%%%%%%%%%%%%%%%%%%%%%%%%%%%%%%
Finally, for simplicity, we write
\[
\delta = d^{2m} \quad\text{and}\quad \zeta = m d^{6m(\tfrac{2h}{d}+3)},
\]
by which these bounds obtained can be rewritten as
\[
\begin{array}{llll}
  D\le \delta^2,        &  \quad  M \le \delta, & \quad d_P\le m \e^{6\delta\zeta},      &  \quad    \hat{\alpha} \le \e^{2\delta\zeta},   \\
  \h(\alpha) \le 5\delta^2\zeta,   &  \quad  \h(\beta) \le m h, & \quad
  \ln|\Delta_\beta|\le 2\zeta.     & 
\end{array}   
\]
Then, by substituting all these bounds into Theorem~\ref{thm:EffLW}, we have 
%the lower bounds for $|\chi''|$ as
\[
\begin{aligned}
\ln\, |\chi''|
\ge - r m^{\delta}\e^{6\delta^2\zeta}
\big( 
m h + \tfrac{39}{164}\zeta  + \e^{R}\big), 
\end{aligned}
\]
in which
\[
 \begin{aligned}
& R=r' m^{\delta}\e^{6\delta^2\zeta} + r'' m^{\delta}\e^{6\delta^2\zeta} (\ln m + 6\delta\zeta) + 72 \e^{2\delta\zeta} ,\\
%%% 
& r =  82(\tfrac{9}{2})^\delta \delta^{3\delta +2},\\
%%%
& r' =  12 (\tfrac{9}{2})^\delta \delta^{3\delta +2}
+ 16(1+6 \delta^2)(\tfrac{9}{2})^\delta \delta^{3\delta} \ln(9\delta^3) \\
& \qquad + 80(\tfrac{9}{2})^\delta \delta^{3\delta +4}
(1+6\delta) \zeta,\\
%%%
& r''=  (1+6\delta^2)16(\tfrac{9}{2})^\delta \delta^{3\delta}.
 \end{aligned}
\]
% Also, we have
% \[
% \ln |\chi'| \ge -\e^{4\zeta}.
% \]
% Then, by combining the lower bounds for
% $|\chi'|$ and $|\chi''|$,
Combining the lower bound  $\ln \, |\chi'| \ge -\e^{8\delta\zeta}$
in Inequality~\eqref{eq:ld1},
we immediately complete the proof of Main Result A.

\section{Proof of Main Result B}
In this section, for algebraic numbers $\alpha_i$ and $\beta_i$ of bounded degrees and heights,
we establish a nontrivial upper bound that
there exists a nonzero $\lambda=\beta_1\e^{\alpha_1}+\cdots+\beta_m\e^{\alpha_m}$ whose absolute value is bounded from above by it.
We construct such $\lambda$ by the difference between 
two distinct linear forms, then the upper bound can be given via Dirichlet's pigeonhole principle.
%the minimal distance between $\lambda_1$ and $\lambda_2$ 
% from above, and
% the result is established
% on systematic results of counting algebraic numbers. 
\subsection{Constructing Linear Forms}
% For an algebraic extension $\K$,
% we define
% \[
% A(\K,h)=\{x\in\K| \h(x)\ge h\}
% \]
% as the set of algebraic numbers in
% $\K$ whose Weil height is at most $h$.
Denote by $\A_d(H)$ the set of
algebraic numbers
having degree $d$ over $\Q$ and the absolute multiplicative height at most $H$.
% Here, we have $H=e$
% of an algebraic number,
% i.e., $\H(\alpha)=\e^{\h(\alpha)}$.
% satisfying $\H(\alpha)<H$.
% Let $\A'_d(H)$ be the subset of $\A_d(H)$ containing all the
% elements with absolute value $\le 1$ in $\A_d(H)$.
% % =\A_d(H) \cap \{\alpha\in\A \mid |\alpha|\le 1\}$.
% \begin{corollary}\label{coro:numA}
% Let $\A'_d(H)$ be the set of algebraic numbers with the
% degree $d$, the maximal absolute multiplicative height $H$,
% and the maximal absolute value $1$.
% Then
% \[
% | \A'_d(H) | > \tfrac{1}{2} 6^{-d(d+1)}H^{d(d+1)} \quad \text{when}\;  H^d \ge 2.
% \]
% \end{corollary}
Now, we consider the set $\Lambda$ which consists of all the linear forms $\sum_i\beta_i\e^{\alpha_i}$ satisfying the following conditions:
\begin{itemize}
    \item[$\cdot$] the linear form contains at most $\ell$ terms,
    % \item[$\cdot$] $\alpha_i\in \A'_{d_1} (H_1)$
    % and $\beta_i\in \A'_{d_2} (H_2)$.
    \item[$\cdot$] any $\alpha_i$ and $\beta_i$ come from
    $\A_{d_1} (H_1)$ and $\A_{d_2} (H_2)$, respectively,
    \item[$\cdot$] the absolute values of $\alpha_i$ and $\beta_i$ are less than or equal to $1$.
    % \item each exponent $\alpha_i$ (resp. coefficients $\beta_i$) is in $\A'_{d_1} (H_1)$ (resp. $\A'_{d_2} (H_2)$);
    % \item each coefficients $\beta_i$ is in $\A'_{d_2} (H_2)$.
\end{itemize}
Let $\A'_d(H)$ be the set of algebraic numbers with 
degree $d$, the maximal absolute multiplicative height $H$,
and the maximal absolute value $1$.
Then, we can count the number of the elements in $\Lambda$.

Consider the linear forms contain exactly $k$ terms.
Noting that the exponents $\alpha_1,\ldots,\alpha_k$
in the linear forms are distinct, 
they can be chosen by $k$-combinations of the elements in $\A'_{d_1}(H_1)$.
The coefficient of each $\e^{\alpha_i}$ can be chosen
as any nonzero element in $\A'_{d_2}(H_2)$.
So we derive that
\begin{equation}\label{eq:numLmd}
|\Lambda|= \sum_{k=1}^{\ell} N_2 ^k \binom{N_1}{k},
\end{equation}
in which $N_1=|\A'_{d_1} (H_1)|$
and $N_2=|\A'_{d_2} (H_2)\setminus\{0\}|$.
Then, with the assumption that $N_1 \ge \ell+1$ and $N_2 \ge 1$, we have
\begin{equation}\label{eq:boundL}
\begin{aligned}
|\Lambda|
& \ge \sum_{k=1}^{\ell} N_2 ^k \left(\frac{N_1}{k}\right)^k 
\ge 
 \sum_{k=1}^{\ell}  N_2 ^k \left(\frac{N_1}{{\ell}}\right)^k  \\
& =
\left(\frac{N_1N_2}{\ell} \right)^\ell
\cdot \frac{1-(\tfrac{{\ell}}{N_1N_2})^{\ell}}{1-\tfrac{{\ell}}{N_1N_2}} \\
& > 
\left(\frac{N_1N_2}{\ell} \right)^\ell.\\
\end{aligned}
\end{equation}

In order to apply the pigeonhole principle, we should ensure that for any
$\lambda_1,\lambda_2\in\Lambda$,
the difference $\lambda_1 - \lambda_2$ is a linear form satisfying the conditions:
(i) it contains at most $m$ terms, and
(ii) its exponents and coefficients are of degrees $\le d$ and absolute multiplicative heights $\le H$. 
Constructively, we can choose the setting as follows:
\begin{itemize}
    \item[$\cdot$] $\ell = \lfloor \tfrac{m}{2}\rfloor$, since $\lambda_1-\lambda_2$ has at most $2\ell$ term,
    \item[$\cdot$] $d_1 = d$ and $H_1 = H$, since all the exponents in $\lambda_1-\lambda_2$ directly come from $\lambda_1$ and $\lambda_2$,
    \item[$\cdot$]  $d_2=\lfloor \sqrt{d} \rfloor$ and $H_2 =  \sqrt{\tfrac{H}{2}}$, 
    %by Property~\ref{prop:h}, because
    since the coefficient in $\lambda_1-\lambda_2$ may come from the difference between two terms with the same $\e^{\alpha_i}$.
\end{itemize}

% Constructively, we can choose any instances satisfying the inequalities above.
% Thus, we set $d_1=d$, $d_2=\lfloor \sqrt{d} \rfloor$, $H_1=H$, $H_2 =  \sqrt{\tfrac{H}{2}}$, and $N_1$ and $N_2$ are determined.
% Also, we set $\ell=\min\left(\lfloor \tfrac{m}{2} \rfloor ,N_1 \right)$ to ensure that $\binom{N_1}{k}$ is nonzero for $i=1,\ldots,\ell$.

\subsection{Analysing the Bound}
For each linear form $\lambda=\sum_i\beta_i\e^{\alpha_i}$ in $\Lambda$,
we can bound its absolute value as
\[
|\lambda| \le \sum_{i=1}^{\ell} |\beta_i| \e^ {|\alpha_i|} \le  \e\ell \le  \frac{1}{2}\e m,
\]
because $\lambda$ contains at most $\ell\le\tfrac{m}{2}$ terms
and the absolute values of $\alpha_i$ and $\beta_i$
are at most $1$.
Thus, we can define
\[
\Omega=\Big \{x+y\i\in\C \mid x,y\in [-\tfrac{1}{2}\e m, \tfrac{1}{2}\e m] \Big\}
\]
such that $\Lambda \subset \Omega$.
Note that $\Omega$ indicates the square
in the complex plane, which is centered at the origin and has the length $\e m$.
We further split $\Omega$ into $t \times t$ grid such that each cell is a square with the same length $\e m {t^{-1}}$.
Then, it is obvious that each $\lambda\in\Lambda$ will be contained by
one of these cells (see Figure~\ref{fig:pig}).

\begin{figure}[ht]
    \centering
    \includegraphics[width=0.38\textwidth]{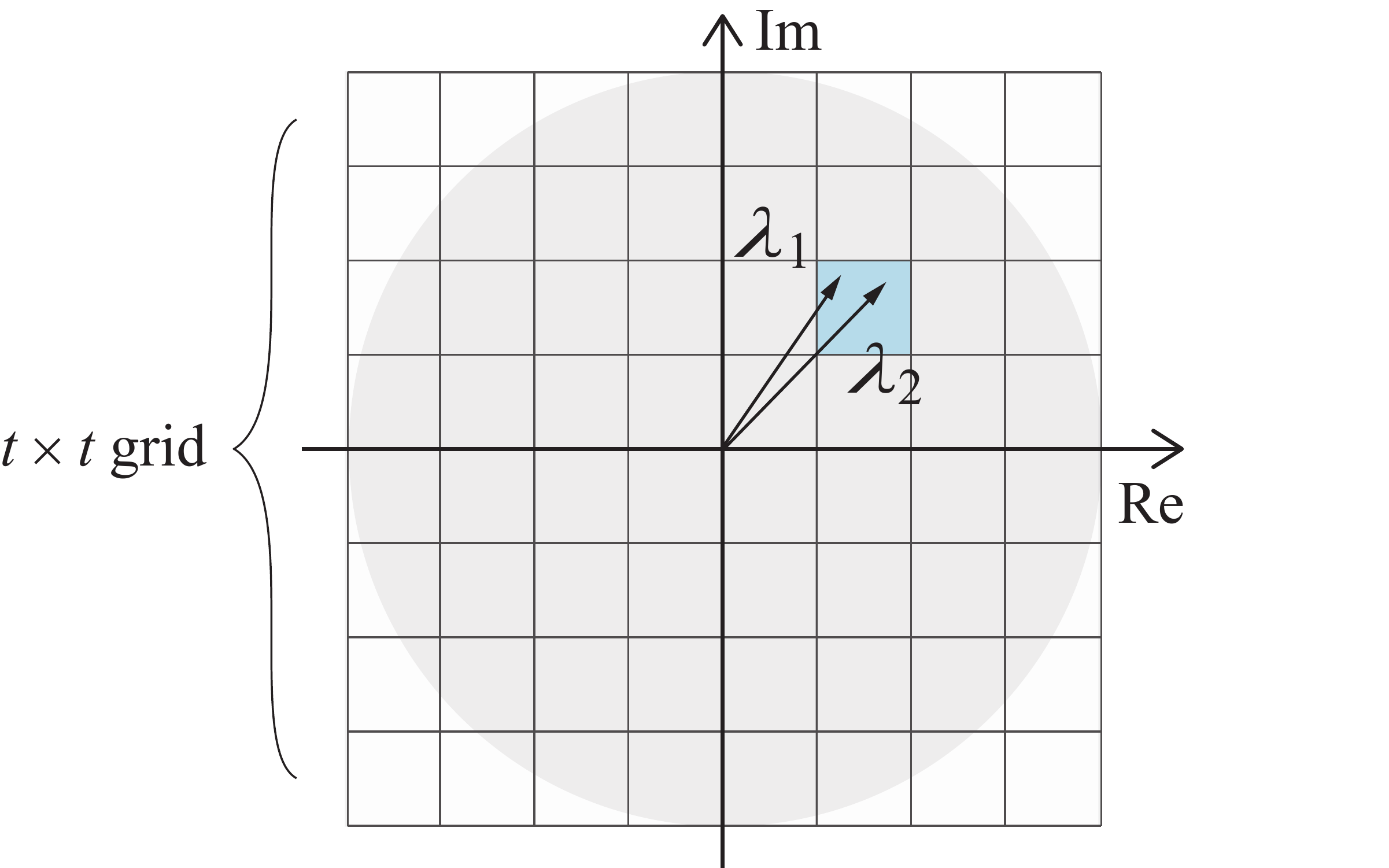}
    \caption{Each element of $\lambda$ is contained by the circle centered at the origin and with radius $\e\ell$. When $|\Lambda|\ge t^2+1$, two distinct elements of $\Lambda$ will be contained by the same cell.}
    \label{fig:pig}
\end{figure}

By Dirichlet's pigeonhole principle, if $|\Lambda|\ge t^2 +1$ holds,
then we have that there exists at least one cell containing more than one element of $\Lambda$. In other words,
there exist $\lambda_1,\lambda_2\in\Lambda$ such that
\begin{equation}\label{eq:elt}
|\lambda_1 - \lambda_2| \le \frac{\sqrt2 \e m}{t}.
\end{equation}
Thus, by Inequality~\eqref{eq:boundL}, let
\[ t= \left\lfloor \sqrt{\left(\frac{N_1N_2}{\ell} \right)^\ell} \right\rfloor
\]
such that $t^2$ is a positive integer less than $|\Lambda|$.

Note the fact that $\ln x-\ln \lfloor x \rfloor < 1$ for any $x\ge 1$.
Then, by Inequality~\eqref{eq:elt}, we immediately obtain the upper bound as
\[
\begin{aligned}
\ln|\lambda_1-\lambda_2|
& \le  - \ln t + \ln (\sqrt 2 \e m) \\
& \le - \ln \sqrt{\left(\frac{N_1N_2}{\ell} \right)^\ell} + \ln (\sqrt 2 \e  m) + 1\\
& \le -\frac{1}{2}\ell\ln \frac{N_1N_2}{\ell}
+\ln{m} + \ln {\sqrt{2}} + 2.
\end{aligned}
\]
% Similarly, for $\ell = N_1N_2$, we have
% \[
% \ln|\lambda_1-\lambda_2| 
% \le  -\frac{1}{2}\ln \ell +\frac{1}{2}\ell\ln{\ell}
% +\ln{\ell} + \ln 2\sqrt{2} + 2.
% \]
To make this upper bound explicit,
the remaining issue is to provide the lower bounds for the numbers of elements in $\A'_{d_1} (H_1)$ and $\A'_{d_2} (H_2)$, respectively.

As a case of one-dimensional of
Northcott's theorem~\cite{northcott1949inequality},
it is well known that $\A_d(H)$
contains finitely many elements,
and counting these elements is also an
important problem in number theory.
In 1993, Schmidt~\cite{schmidt1993northcott} first proved the
explicit lower bound for the number of the elements in $\A_d(H)$.

\begin{lemma}[Schmidt, Theorem of \cite{schmidt1993northcott}]\label{lem:numalg}
Let $\A_d(H)$ be the set of algebraic numbers with
degree $d$ and the maximal absolute multiplicative height $H$.
Then
\[
|\A_d(H)| > 6^{-d(d+1)}H^{d(d+1)} \quad \text{when}\;  H^d \ge 2.
%(\tfrac{H}{6})^{d(d+1)} 
\]
\end{lemma}

In our situation, we focus on the elements in $\A_d(H)$ whose
absolute values are less than or equal to $1$.
Consider any nonzero algebraic number $\alpha\in\A_d(H)$ with $|\alpha|>1$,
and denote by $f(x)$ the defining polynomial of $\alpha$.
It is not hard to check that $\alpha^{-1}$ is a root of the polynomial $g\in\Z[x]$ of degree $d$ where
\[
g(x)=x^d f(\tfrac{1}{x}).
\]
Also, by Property~\ref{prop:h}, we further have
$\h(\alpha^{-1})=\h(\alpha)$,
that implies $\alpha^{-1}$ is also contained in $\A_d(H)$.
With the fact that $|\alpha^{-1}|=|\alpha|^{-1}$, we immediately infer that
$| \A'_d(H) | \ge \tfrac{1}{2} | \A_d(H) |$. 
Thus, according to Theorem~\ref{lem:numalg}, we finally have
\[
| \A'_d(H) | > \tfrac{1}{2} 6^{-d(d+1)}H^{d(d+1)} \quad \text{when}\;  H^d \ge 2.
\]
Therefore, we can write 
\[
\begin{array}{l}
    n_1 = \tfrac{1}{2} 6^{-d(d+1)}H^{d(d+1)},\\
    n_2 = \tfrac{1}{2} 6^{-(\sqrt{d}-1)\sqrt{d}} \sqrt{\tfrac{H}{2}}^{(\sqrt{d}-1)\sqrt{d}},
\end{array}
\]
such that $N_1$ and $N_2$ are positive integers greater than
$n_1$ and $n_2$, respectively.
This complete the proof of Main Result B.

\begin{remark}
There are some conditions here to make the result hold.
On the one hand, we need $n_1\ge \ell$ and $n_2\ge 1$, which imply the assumptions $N_1 \ge \ell + 1$ and $N_2 \ge 1$ for Inequality~\eqref{eq:boundL}. The aim is to ensure that we have enough algebraic numbers to construct the linear forms.
    Besides, they can also imply
    the conditions $H_1^{d_1}\ge 2$ and $H_2^{d_2} \ge 2$ required as in
    Lemma~\ref{lem:numalg}.
    % requires the condition that $H^d\ge 2$, which implies the length of the algebraic numbers $\ge 2$. In other words, it ensures that $\A_d(H)\setminus\{0\}$ is nonempty. However, this condition can be omitted in our result since it has already been implied by $n_1\ge \ell$

On the other hand, we have $m \ge 2$, which comes from the setting that
$\ell = \lfloor \tfrac{m}{2}\rfloor$ to ensures that $\ell$ is nonzero.  
However, with the fact that $\lim_{\ell\to 0} \ell\ln\ell = 0$, the upper bound we obtained is also clearly valid for the case $m=1$. So, for simplicity, we omit this condition in the result.
Actually, the situation of the linear forms containing one term is trivial because $\beta \e ^{\alpha}$ reaches its minimal absolute value when both $|\alpha|$ and $|\beta|$ reach their minimal values, respectively.
 \end{remark}

\bibliographystyle{ACM-Reference-Format}
% \bibliography{main}

\begin{thebibliography}{41}

%%% ====================================================================
%%% NOTE TO THE USER: you can override these defaults by providing
%%% customized versions of any of these macros before the \bibliography
%%% command.  Each of them MUST provide its own final punctuation,
%%% except for \shownote{}, \showDOI{}, and \showURL{}.  The latter two
%%% do not use final punctuation, in order to avoid confusing it with
%%% the Web address.
%%%
%%% To suppress output of a particular field, define its macro to expand
%%% to an empty string, or better, \unskip, like this:
%%%
%%% \newcommand{\showDOI}[1]{\unskip}   % LaTeX syntax
%%%
%%% \def \showDOI #1{\unskip}           % plain TeX syntax
%%%
%%% ====================================================================

\ifx \showCODEN    \undefined \def \showCODEN     #1{\unskip}     \fi
\ifx \showDOI      \undefined \def \showDOI       #1{#1}\fi
\ifx \showISBNx    \undefined \def \showISBNx     #1{\unskip}     \fi
\ifx \showISBNxiii \undefined \def \showISBNxiii  #1{\unskip}     \fi
\ifx \showISSN     \undefined \def \showISSN      #1{\unskip}     \fi
\ifx \showLCCN     \undefined \def \showLCCN      #1{\unskip}     \fi
\ifx \shownote     \undefined \def \shownote      #1{#1}          \fi
\ifx \showarticletitle \undefined \def \showarticletitle #1{#1}   \fi
\ifx \showURL      \undefined \def \showURL       {\relax}        \fi
% The following commands are used for tagged output and should be
% invisible to TeX
\providecommand\bibfield[2]{#2}
\providecommand\bibinfo[2]{#2}
\providecommand\natexlab[1]{#1}
\providecommand\showeprint[2][]{arXiv:#2}

\bibitem[Ably(1994)]%
        {ably1994version}
\bibfield{author}{\bibinfo{person}{Mohammed Ably}.}
  \bibinfo{year}{1994}\natexlab{}.
\newblock \showarticletitle{Une version quantitative du th{\'e}or{\`e}me de
  Lindemann-Weierstrass}.
\newblock \bibinfo{journal}{\emph{Acta Arithmetica}}  \bibinfo{volume}{67}
  (\bibinfo{year}{1994}), \bibinfo{pages}{29--45}.
\newblock


\bibitem[Almagor et~al\mbox{.}(2018)]%
        {almagor2018minimal}
\bibfield{author}{\bibinfo{person}{Shaull Almagor}, \bibinfo{person}{Dmitry
  Chistikov}, \bibinfo{person}{Jo{\"e}l Ouaknine}, {and} \bibinfo{person}{James
  Worrell}.} \bibinfo{year}{2018}\natexlab{}.
\newblock \showarticletitle{O-minimal invariants for linear loops}. In
  \bibinfo{booktitle}{\emph{45th International Colloquium on Automata,
  Languages, and Programming}}. Schloss Dagstuhl, \bibinfo{pages}{1--14}.
\newblock


\bibitem[Anai and Weispfenning(2000)]%
        {anai2000deciding}
\bibfield{author}{\bibinfo{person}{Hirokazu Anai} {and} \bibinfo{person}{Volker
  Weispfenning}.} \bibinfo{year}{2000}\natexlab{}.
\newblock \showarticletitle{Deciding linear-trigonometric problems}. In
  \bibinfo{booktitle}{\emph{25th international symposium on Symbolic and
  algebraic computation}}. \bibinfo{pages}{14--22}.
\newblock


\bibitem[Ax(1971)]%
        {ax1971schanuel}
\bibfield{author}{\bibinfo{person}{James Ax}.} \bibinfo{year}{1971}\natexlab{}.
\newblock \showarticletitle{On Schanuel's conjectures}.
\newblock \bibinfo{journal}{\emph{Annals of mathematics}} \bibinfo{volume}{93},
  \bibinfo{number}{2} (\bibinfo{year}{1971}), \bibinfo{pages}{252--268}.
\newblock


\bibitem[Aziz et~al\mbox{.}(1996)]%
        {AzizSSB96}
\bibfield{author}{\bibinfo{person}{Adnan Aziz}, \bibinfo{person}{Kumud Sanwal},
  \bibinfo{person}{Vigyan Singhal}, {and} \bibinfo{person}{Robert~K. Brayton}.}
  \bibinfo{year}{1996}\natexlab{}.
\newblock \showarticletitle{Verifying Continuous Time Markov Chains}. In
  \bibinfo{booktitle}{\emph{Computer Aided Verification, 8th International
  Conference}}, Vol.~\bibinfo{volume}{1102}. \bibinfo{pages}{269--276}.
\newblock


\bibitem[Baker(1990)]%
        {baker1990transcendental}
\bibfield{author}{\bibinfo{person}{Alan Baker}.}
  \bibinfo{year}{1990}\natexlab{}.
\newblock \bibinfo{booktitle}{\emph{Transcendental number theory}}.
\newblock \bibinfo{publisher}{Cambridge university press}.
\newblock


\bibitem[Basu et~al\mbox{.}(2003)]%
        {saugata2003geometry}
\bibfield{author}{\bibinfo{person}{Saugata Basu}, \bibinfo{person}{Richard
  Pollack}, {and} \bibinfo{person}{Marie-Françoise Roy}.}
  \bibinfo{year}{2003}\natexlab{}.
\newblock \bibinfo{booktitle}{\emph{Algorithms in Real Algebraic Geometry}}.
\newblock \bibinfo{publisher}{Springer}.
\newblock


\bibitem[Bell et~al\mbox{.}(2010)]%
        {bell2010continuous}
\bibfield{author}{\bibinfo{person}{Paul~C Bell}, \bibinfo{person}{Jean-Charles
  Delvenne}, \bibinfo{person}{Rapha{\"e}l~M Jungers}, {and}
  \bibinfo{person}{Vincent~D Blondel}.} \bibinfo{year}{2010}\natexlab{}.
\newblock \showarticletitle{The continuous skolem-pisot problem}.
\newblock \bibinfo{journal}{\emph{Theoretical Computer Science}}
  \bibinfo{volume}{411}, \bibinfo{number}{40-42} (\bibinfo{year}{2010}),
  \bibinfo{pages}{3625--3634}.
\newblock


\bibitem[Canny(1988)]%
        {canny1988complexity}
\bibfield{author}{\bibinfo{person}{John Canny}.}
  \bibinfo{year}{1988}\natexlab{}.
\newblock \bibinfo{booktitle}{\emph{The complexity of robot motion planning}}.
\newblock \bibinfo{publisher}{MIT press}.
\newblock


\bibitem[Cantor(1874)]%
        {cantor1874ueber}
\bibfield{author}{\bibinfo{person}{Georg Cantor}.}
  \bibinfo{year}{1874}\natexlab{}.
\newblock \showarticletitle{Ueber eine Eigenschaft des Inbegriffs aller reellen
  algebraischen Zahlen.}
\newblock \bibinfo{journal}{\emph{Journal f{\"u}r die reine und angewandte
  Mathematik}}  \bibinfo{volume}{77} (\bibinfo{year}{1874}),
  \bibinfo{pages}{258--262}.
\newblock


\bibitem[Chen et~al\mbox{.}(2015)]%
        {chen2015continuous}
\bibfield{author}{\bibinfo{person}{Taolue Chen}, \bibinfo{person}{Nengkun Yu},
  {and} \bibinfo{person}{Tingting Han}.} \bibinfo{year}{2015}\natexlab{}.
\newblock \showarticletitle{Continuous-time orbit problems are decidable in
  polynomial-time}.
\newblock \bibinfo{journal}{\emph{Inform. Process. Lett.}}
  \bibinfo{volume}{115}, \bibinfo{number}{1} (\bibinfo{year}{2015}),
  \bibinfo{pages}{11--14}.
\newblock


\bibitem[Chistikov et~al\mbox{.}(2020)]%
        {chistikov2020big}
\bibfield{author}{\bibinfo{person}{Dmitry Chistikov}, \bibinfo{person}{Stefan
  Kiefer}, \bibinfo{person}{Andrzej~S Murawski}, {and} \bibinfo{person}{David
  Purser}.} \bibinfo{year}{2020}\natexlab{}.
\newblock \showarticletitle{The Big-O Problem for Labelled Markov Chains and
  Weighted Automata}. In \bibinfo{booktitle}{\emph{31st International
  Conference on Concurrency Theory}}.
\newblock


\bibitem[Chonev et~al\mbox{.}(2016)]%
        {chonev2016skolem}
\bibfield{author}{\bibinfo{person}{Ventsislav Chonev},
  \bibinfo{person}{Jo{\"e}l Ouaknine}, {and} \bibinfo{person}{James Worrell}.}
  \bibinfo{year}{2016}\natexlab{}.
\newblock \showarticletitle{{On the Skolem Problem for Continuous Linear
  Dynamical Systems}}. In \bibinfo{booktitle}{\emph{43rd International
  Colloquium on Automata, Languages, and Programming}},
  Vol.~\bibinfo{volume}{55}. \bibinfo{pages}{100:1--100:13}.
\newblock


\bibitem[Cohen et~al\mbox{.}(1993)]%
        {cohen1993course}
\bibfield{author}{\bibinfo{person}{Henri Cohen}, \bibinfo{person}{Henry Cohen},
  {and} \bibinfo{person}{Henri Cohen}.} \bibinfo{year}{1993}\natexlab{}.
\newblock \bibinfo{booktitle}{\emph{A course in computational algebraic number
  theory}}. Vol.~\bibinfo{volume}{8}.
\newblock \bibinfo{publisher}{Springer-Verlag Berlin}.
\newblock


\bibitem[Collins and Horowitz(1974)]%
        {collins1974minimum}
\bibfield{author}{\bibinfo{person}{George~E Collins} {and}
  \bibinfo{person}{Ellis Horowitz}.} \bibinfo{year}{1974}\natexlab{}.
\newblock \showarticletitle{The minimum root separation of a polynomial}.
\newblock \bibinfo{journal}{\emph{mathematics of computation}}
  \bibinfo{volume}{28}, \bibinfo{number}{126} (\bibinfo{year}{1974}),
  \bibinfo{pages}{589--597}.
\newblock


\bibitem[Feng and Zhang(2017)]%
        {feng2017precisely}
\bibfield{author}{\bibinfo{person}{Yuan Feng} {and} \bibinfo{person}{Lijun
  Zhang}.} \bibinfo{year}{2017}\natexlab{}.
\newblock \showarticletitle{Precisely deciding CSL formulas through approximate
  model checking for CTMCs}.
\newblock \bibinfo{journal}{\emph{J. Comput. System Sci.}}
  \bibinfo{volume}{89} (\bibinfo{year}{2017}), \bibinfo{pages}{361--371}.
\newblock


\bibitem[Garc{\'\i}a-Marco et~al\mbox{.}(2018)]%
        {garcia2018polynomial}
\bibfield{author}{\bibinfo{person}{Ignacio Garc{\'\i}a-Marco},
  \bibinfo{person}{Pascal Koiran}, {and} \bibinfo{person}{Timoth{\'e}e
  Pecatte}.} \bibinfo{year}{2018}\natexlab{}.
\newblock \showarticletitle{Polynomial equivalence problems for sum of affine
  powers}. In \bibinfo{booktitle}{\emph{Proceedings of the 2018 ACM
  International Symposium on Symbolic and Algebraic Computation}}.
  \bibinfo{pages}{303--310}.
\newblock


\bibitem[Hermite(1874)]%
        {hermite1874fonction}
\bibfield{author}{\bibinfo{person}{Charles Hermite}.}
  \bibinfo{year}{1874}\natexlab{}.
\newblock \bibinfo{booktitle}{\emph{Sur la fonction exponentielle}}.
\newblock \bibinfo{publisher}{Gauthier-Villars}.
\newblock


\bibitem[Huang et~al\mbox{.}(2018)]%
        {huang2018positive}
\bibfield{author}{\bibinfo{person}{Cheng-Chao Huang}, \bibinfo{person}{Jing-Cao
  Li}, \bibinfo{person}{Ming Xu}, {and} \bibinfo{person}{Zhi-Bin Li}.}
  \bibinfo{year}{2018}\natexlab{}.
\newblock \showarticletitle{Positive root isolation for poly-powers by
  exclusion and differentiation}.
\newblock \bibinfo{journal}{\emph{Journal of Symbolic Computation}}
  \bibinfo{volume}{85} (\bibinfo{year}{2018}), \bibinfo{pages}{148--169}.
\newblock


\bibitem[Lenstra(1983)]%
        {lenstra1983factoring}
\bibfield{author}{\bibinfo{person}{Arjen~K Lenstra}.}
  \bibinfo{year}{1983}\natexlab{}.
\newblock \showarticletitle{Factoring polynomials over algebraic number
  fields}. In \bibinfo{booktitle}{\emph{European Conference on Computer
  Algebra}}. Springer, \bibinfo{pages}{245--254}.
\newblock


\bibitem[Lenstra et~al\mbox{.}(1982)]%
        {lenstra1982factoring}
\bibfield{author}{\bibinfo{person}{Arjen~K Lenstra},
  \bibinfo{person}{Hendrik~Willem Lenstra}, {and}
  \bibinfo{person}{L{\'a}szl{\'o} Lov{\'a}sz}.}
  \bibinfo{year}{1982}\natexlab{}.
\newblock \showarticletitle{Factoring polynomials with rational coefficients}.
\newblock \bibinfo{journal}{\emph{Mathematische annalen}}
  \bibinfo{volume}{261}, \bibinfo{number}{ARTICLE} (\bibinfo{year}{1982}),
  \bibinfo{pages}{515--534}.
\newblock

\vfill\eject

\bibitem[Lindemann(1882)]%
        {lindemann1882ueber}
\bibfield{author}{\bibinfo{person}{Ferdinand Lindemann}.}
  \bibinfo{year}{1882}\natexlab{}.
\newblock \showarticletitle{Ueber die Zahl $\pi$.*}.
\newblock \bibinfo{journal}{\emph{Math. Ann.}} \bibinfo{volume}{20},
  \bibinfo{number}{2} (\bibinfo{year}{1882}), \bibinfo{pages}{213--225}.
\newblock


\bibitem[Liouville(1844)]%
        {liouville1844classes}
\bibfield{author}{\bibinfo{person}{Joseph Liouville}.}
  \bibinfo{year}{1844}\natexlab{}.
\newblock \showarticletitle{Sur des classes tr{\`e}s-{\'e}tendues de
  quantit{\'e}s dont la valeur n'est ni alg{\'e}brique ni m{\^e}me
  r{\'e}ductible {\`a} des irrationnelles alg{\'e}briques}.
\newblock \bibinfo{journal}{\emph{CR Acad. Sci. Paris}}  \bibinfo{volume}{18}
  (\bibinfo{year}{1844}), \bibinfo{pages}{883--885}.
\newblock


\bibitem[Loos(1982)]%
        {loos1982computing}
\bibfield{author}{\bibinfo{person}{R{\"u}diger Loos}.}
  \bibinfo{year}{1982}\natexlab{}.
\newblock \showarticletitle{Computing in algebraic extensions}.
\newblock In \bibinfo{booktitle}{\emph{Computer algebra}}.
  \bibinfo{publisher}{Springer}, \bibinfo{pages}{173--187}.
\newblock


\bibitem[Mahler(1932)]%
        {mahler1932approximation}
\bibfield{author}{\bibinfo{person}{Kurt Mahler}.}
  \bibinfo{year}{1932}\natexlab{}.
\newblock \showarticletitle{Zur Approximation der Exponentialfunktion und des
  Logarithmus. Teil I. (1932), 118-136.}
\newblock  (\bibinfo{year}{1932}).
\newblock


\bibitem[Mahler(1960)]%
        {mahler1960application}
\bibfield{author}{\bibinfo{person}{Kurt Mahler}.}
  \bibinfo{year}{1960}\natexlab{}.
\newblock \showarticletitle{An application of Jensen's formula to polynomials}.
\newblock \bibinfo{journal}{\emph{Mathematika}} \bibinfo{volume}{7},
  \bibinfo{number}{2} (\bibinfo{year}{1960}), \bibinfo{pages}{98--100}.
\newblock


\bibitem[Mahler et~al\mbox{.}(1964)]%
        {mahler1964inequality}
\bibfield{author}{\bibinfo{person}{Kurt Mahler} {et~al\mbox{.}}}
  \bibinfo{year}{1964}\natexlab{}.
\newblock \showarticletitle{An inequality for the discriminant of a
  polynomial.}
\newblock \bibinfo{journal}{\emph{Michigan Mathematical Journal}}
  \bibinfo{volume}{11}, \bibinfo{number}{3} (\bibinfo{year}{1964}),
  \bibinfo{pages}{257--262}.
\newblock


\bibitem[Majumdar et~al\mbox{.}(2020)]%
        {majumdar2020decidability}
\bibfield{author}{\bibinfo{person}{Rupak Majumdar}, \bibinfo{person}{Mahmoud
  Salamati}, {and} \bibinfo{person}{Sadegh Soudjani}.}
  \bibinfo{year}{2020}\natexlab{}.
\newblock \showarticletitle{On Decidability of Time-Bounded Reachability in
  CTMDPs}. In \bibinfo{booktitle}{\emph{47th International Colloquium on
  Automata, Languages, and Programming}}.
\newblock


\bibitem[Majumdar and Soudjani(2021)]%
        {majumdar2021computability}
\bibfield{author}{\bibinfo{person}{Rupak Majumdar} {and}
  \bibinfo{person}{Sadegh Soudjani}.} \bibinfo{year}{2021}\natexlab{}.
\newblock \showarticletitle{The computability of LQR and LQG control}. In
  \bibinfo{booktitle}{\emph{24th International Conference on Hybrid Systems:
  Computation and Control}}. \bibinfo{pages}{1--7}.
\newblock


\bibitem[McCallum and Weispfenning(2012)]%
        {mccallum2012deciding}
\bibfield{author}{\bibinfo{person}{Scott McCallum} {and}
  \bibinfo{person}{Volker Weispfenning}.} \bibinfo{year}{2012}\natexlab{}.
\newblock \showarticletitle{Deciding polynomial-transcendental problems}.
\newblock \bibinfo{journal}{\emph{Journal of Symbolic Computation}}
  \bibinfo{volume}{47}, \bibinfo{number}{1} (\bibinfo{year}{2012}),
  \bibinfo{pages}{16--31}.
\newblock


\bibitem[Northcott(1949)]%
        {northcott1949inequality}
\bibfield{author}{\bibinfo{person}{Douglas~G Northcott}.}
  \bibinfo{year}{1949}\natexlab{}.
\newblock \showarticletitle{An inequality in the theory of arithmetic on
  algebraic varieties}. In \bibinfo{booktitle}{\emph{Mathematical Proceedings
  of the Cambridge Philosophical Society}}, Vol.~\bibinfo{volume}{45}.
  Cambridge University Press, \bibinfo{pages}{502--509}.
\newblock


\bibitem[Philippon(1986)]%
        {philippon1986criteres}
\bibfield{author}{\bibinfo{person}{Patrice Philippon}.}
  \bibinfo{year}{1986}\natexlab{}.
\newblock \showarticletitle{Crit{\`e}res pour l'ind{\'e}pendance
  alg{\'e}brique}.
\newblock \bibinfo{journal}{\emph{Publications Math{\'e}matiques de
  l'IH{\'E}S}}  \bibinfo{volume}{64} (\bibinfo{year}{1986}),
  \bibinfo{pages}{5--52}.
\newblock


\bibitem[Schmidt(1993)]%
        {schmidt1993northcott}
\bibfield{author}{\bibinfo{person}{Wolfgang~M Schmidt}.}
  \bibinfo{year}{1993}\natexlab{}.
\newblock \showarticletitle{Northcott's theorem on heights I. A general
  estimate}.
\newblock \bibinfo{journal}{\emph{Monatshefte f{\"u}r Mathematik}}
  \bibinfo{volume}{115}, \bibinfo{number}{1} (\bibinfo{year}{1993}),
  \bibinfo{pages}{169--181}.
\newblock


\bibitem[Sert(1999)]%
        {sert1999version}
\bibfield{author}{\bibinfo{person}{Alain Sert}.}
  \bibinfo{year}{1999}\natexlab{}.
\newblock \showarticletitle{Une version effective du th{\'e}or{\`e}me de
  Lindemann--Weierstrass par les d{\'e}terminants d'interpolation}.
\newblock \bibinfo{journal}{\emph{Journal of Number Theory}}
  \bibinfo{volume}{76}, \bibinfo{number}{1} (\bibinfo{year}{1999}),
  \bibinfo{pages}{94--119}.
\newblock


\bibitem[Siegel(2014)]%
        {siegel2014einige}
\bibfield{author}{\bibinfo{person}{Carl~L Siegel}.}
  \bibinfo{year}{2014}\natexlab{}.
\newblock \showarticletitle{{\"U}ber einige anwendungen diophantischer
  approximationen}.
\newblock In \bibinfo{booktitle}{\emph{On Some Applications of Diophantine
  Approximations}}. \bibinfo{publisher}{Springer}, \bibinfo{pages}{81--138}.
\newblock


\bibitem[Strzebonski(2009)]%
        {strzebonski2009real}
\bibfield{author}{\bibinfo{person}{Adam Strzebonski}.}
  \bibinfo{year}{2009}\natexlab{}.
\newblock \showarticletitle{Real root isolation for tame elementary functions}.
  In \bibinfo{booktitle}{\emph{34th international symposium on Symbolic and
  algebraic computation}}. \bibinfo{pages}{341--350}.
\newblock


\bibitem[Waldschmidt(2013)]%
        {waldschmidt2013diophantine}
\bibfield{author}{\bibinfo{person}{Michel Waldschmidt}.}
  \bibinfo{year}{2013}\natexlab{}.
\newblock \bibinfo{booktitle}{\emph{Diophantine approximation on linear
  algebraic groups: transcendence properties of the exponential function in
  several variables}}. Vol.~\bibinfo{volume}{326}.
\newblock \bibinfo{publisher}{Springer Science \& Business Media}.
\newblock


\bibitem[Weierstrass(1885)]%
        {weierstrass1885lindemann}
\bibfield{author}{\bibinfo{person}{Karl Weierstrass}.}
  \bibinfo{year}{1885}\natexlab{}.
\newblock \showarticletitle{Zu Lindemann’s Abhandlung:``{\"U}ber die
  Ludolph’sche Zahl''}.
\newblock \bibinfo{journal}{\emph{Sitzungber. K{\"o}nigl. Preuss. Akad.
  Wissensch}} \bibinfo{number}{2} (\bibinfo{year}{1885}),
  \bibinfo{pages}{1067–1086}.
\newblock


\bibitem[Xu et~al\mbox{.}(2021)]%
        {Xu2021QCTMC}
\bibfield{author}{\bibinfo{person}{Ming Xu}, \bibinfo{person}{Jingyi Mei},
  \bibinfo{person}{Ji Guan}, {and} \bibinfo{person}{Nengkun Yu}.}
  \bibinfo{year}{2021}\natexlab{}.
\newblock \showarticletitle{Model Checking Quantum Continuous-Time Markov
  Chains}. In \bibinfo{booktitle}{\emph{32nd International Conference on
  Concurrency Theory}}, Vol.~\bibinfo{volume}{203}.
  \bibinfo{pages}{13:1--13:17}.
\newblock


\bibitem[Xu et~al\mbox{.}(2016)]%
        {xu2016multiphase}
\bibfield{author}{\bibinfo{person}{Ming Xu}, \bibinfo{person}{Lijun Zhang},
  \bibinfo{person}{David~N Jansen}, \bibinfo{person}{Huibiao Zhu}, {and}
  \bibinfo{person}{Zongyuan Yang}.} \bibinfo{year}{2016}\natexlab{}.
\newblock \showarticletitle{Multiphase until formulas over Markov reward
  models: An algebraic approach}.
\newblock \bibinfo{journal}{\emph{Theoretical Computer Science}}
  \bibinfo{volume}{611} (\bibinfo{year}{2016}), \bibinfo{pages}{116--135}.
\newblock


\bibitem[Yokoyama et~al\mbox{.}(1989)]%
        {yokoyama1989computing}
\bibfield{author}{\bibinfo{person}{Kazuhiro Yokoyama},
  \bibinfo{person}{Masayuki Noro}, {and} \bibinfo{person}{Taku Takeshima}.}
  \bibinfo{year}{1989}\natexlab{}.
\newblock \showarticletitle{Computing primitive elements of extension fields}.
\newblock \bibinfo{journal}{\emph{Journal of symbolic computation}}
  \bibinfo{volume}{8}, \bibinfo{number}{6} (\bibinfo{year}{1989}),
  \bibinfo{pages}{553--580}.
\newblock


\end{thebibliography}

%%% -*-BibTeX-*-
%%% Do NOT edit. File created by BibTeX with style
%%% ACM-Reference-Format-Journals [18-Jan-2012].

\end{document}